\documentclass[12pt,twoside,reqno]{amsart}

\usepackage{tikz}
\usepackage{amsmath, verbatim}
\usepackage{amssymb,amsfonts,mathrsfs}
\usepackage[colorlinks=true,linkcolor=blue,citecolor=blue]{hyperref}
\usepackage[driver=pdftex,margin=2.5cm,heightrounded=true,centering]{geometry}

\usepackage{mllocal} 

\newcommand{\scal}{\textup{scal}}
\newcommand{\nmhho}[1]{\left\| #1 \right\|_{2+\A}}
\newcommand{\nmhhok}[1]{\left\| #1 \right\|_{2k+\A}}

\newcommand{\nmho}[1]{\left\| #1  \right\|_{\A}}
\newcommand{\ginit}{g_{\, \textup{init}}}

\DeclareMathOperator{\hhok}{\Lambda^{2k+\A, k+\frac{\A}{2}}(M\times [0,T])}
\DeclareMathOperator{\hhokk}{\Lambda^{2(k+1)+\A, (k+1)+\frac{\A}{2}}(M\times [0,T])}
\DeclareMathOperator{\hhokm}{\Lambda^{2(k\, -\, 1)+\A, (k \, - \, 1)+\frac{\A}{2}}(M\times [0,T])}

\numberwithin{equation}{section}

\begin{document}


\title[Yamabe flow on Manifolds with Edges]
{Yamabe flow on Manifolds with Edges}

\author{Eric Bahuaud}
\address{Department of Mathematics,
Seattle University,
Seattle, WA 98122,
USA}

\email{bahuaude@seattleu.edu}

\author{Boris Vertman}
\address{Mathematisches Institut,
Universit\"at Bonn,
53115 Bonn,
Germany}
\email{vertman@math.uni-bonn.de}
\urladdr{www.math.uni-bonn.de/people/vertman}

\date{\today}

\begin{abstract}
{Let $(M,g)$ be a compact oriented Riemannian manifold with an incomplete edge singularity.
This article shows that it is possible to evolve $g$ by the Yamabe flow within a class of singular edge metrics. 
As the main analytic step we establish parabolic Schauder-type estimates for the heat operator 
on certain H\"older spaces adapted to the singular edge geometry. We apply these estimates to obtain 
local existence for a variety of quasilinear equations, including the Yamabe flow. This provides a setup for a 
subsequent discussion of the Yamabe problem using flow techniques in the singular setting.}
\end{abstract}

\maketitle
\tableofcontents

\section{Introduction and statement of the main result}\label{intro}

On a compact Riemannian manifold, the Yamabe problem asserts that every conformal class of metrics contains a representative of constant scalar curvature.  There are now several approaches to proving this fact.  The first complete proof, commenced in \cite{Yamabe} and continued by \cite{Trudinger, Aubin, Schoen} used the calculus of variations and elliptic partial differential equations; see \cite{LeeParker} for a survey.  Another method, introduced by Hamilton uses the geometric Yamabe flow:
\begin{equation} \label{YF}  
\left\{ \begin{array}{ll} \partial_{t} g &= - \, \scal(g(t)) \cdot g, \\
                            g (0) &= \ginit, \end{array} \right. 
\end{equation}

(and appropriate normalizations) to converge to constant scalar curvature metrics; see \cite{Brendle} for a recent overview.

It is natural to wonder to what extent the Yamabe problem holds in other settings.  There has been recent work in understanding this problem on singular manifolds with conic and more general incomplete edge metrics\footnote{Note that this is not what is commonly called the singular Yamabe problem which asks that given an appropriate closed subset $\Gamma \subset M$ find a \textit{complete} metric of constant scalar curvature on $M \setminus \Gamma$}.  See the work of Akutagawa-Botvinnik \cite{AkutagawaBotvinnik} for the conic case, cf. also Jeffres-Rowlett \cite{JeffresRowlett}.  Further, see Akutagawa-Carron-Mazzeo \cite{ACM} for a study the Yamabe problem on very general stratified spaces.  All of the papers mentioned above attack the problem from an elliptic PDE point of view.  

Regarding a geometric flow approach, the Ricci flow on (two-dimensional) surfaces starting at singular metrics has been studied by various authors; see for example the recent review by Isenberg-Mazzeo-Sesum \cite{IMS:RFI}.  One expects uniqueness of the flow to fail without specifying further boundary restrictions.  Not only is it possible to obtain a Ricci flow that remains singular (see Mazzeo-Rubinstein-Sesum \cite{MRS} and Yin \cite{Yin:RFO}), but one may also obtain a solution to the flow starting at a singular metric that becomes instantaneously complete (see Giesen-Topping \cite{Topping, Topping2}) or smooths out the cone angle (see Simon \cite{MilesSimon}).  

In this paper and its successor we are interested in the Yamabe flow on spaces with incomplete edge metrics that preserves the singular structure.  This is a first step to study the Yamabe problem from the geometric flow perspective within the setup of singular spaces.  Our work is for dimension $m > 2$.

On closed manifolds the short-time existence of the Yamabe flow is a basic consequence of the classical theory for parabolic partial differential equations.  This classical theory may be found for example in the book by Ladyzhenskaya-Solonnikov-Uraltseva \cite{LSU:LAQ}.  In the presence of conical or edge singularities the analysis is complicated by the fact that both the operators and the scalar curvature are singular.  Thus geometric and analytic conditions must be imposed to even make sense of the flow.   

To expand on this a little more, consider as an example a manifold with an isolated rigid conic singularity.  In a neighbourhood of the cone tip, we may express the metric as 
\[ g = dx^2 + x^2 k, \]
where $k$ is a smooth metric on the $f$-dimensional link $F$, and $x$ is the distance function.  A straightforward computation shows that the singular terms in the expansion for the scalar curvature of $g$ are
\[ \scal(g) = x^{-2} \left( \scal(k) - f(f-1) \right) + O(1). \]
This asymptotic calculation still holds at the leading order even if we allow higher order terms in the metric.  As already noted in \cite{AkutagawaBotvinnik}, in order to solve the Yamabe problem in an asymptotic class of singular metrics we must have $\scal(k) = f(f-1)$, which places a topological restriction on the link.  Furthermore, as we discuss below, the Yamabe flow may be written as a quasilinear parabolic equation for a conformal factor involving the (singular) Laplacian of the initial metric.  This equation need not a priori preserve any special structure of the metric, and so an appropriate class of admissible metrics must be described that are preserved under the flow.  Our approach is to translate this problem to a regularity statement for the initial scalar curvature in an appropriate H\"older function space.

We mention that in the setup of isolated conical singularities, existence and regularity of solutions to the inhomogeneous heat equation has recently been addressed in a recent preprint of Behrndt \cite{Beh:OTC}.  Another approach to estimates for conical singularities, and different from the one we give here, is given in the forthcoming paper by Mazzeo-Rubinstein-Sesum \cite{MRS}.

Our main result is as follows; see Theorem \ref{mainthm} below for the precise statement specifying the class of metrics we consider.

\begin{thm}
There exists a solution to the Yamabe flow starting within a class of compact 
Riemannian spaces with admissible simple edge singularities that remains asymptotically admissible for a short-time.
\end{thm}

In order to prove this theorem we will describe appropriate function spaces on which we can obtain estimates for the linear problem.  The remainder of the proof consists of setting up a contraction mapping argument. 

The remainder of the introduction is organized as follows.  In the next subsection we discuss the class of incomplete edge metrics.  We then describe our parabolic H\"older spaces, adapted to the singular geometry, which were also recently used by Donaldson \cite{Donaldson} and Jeffres, Mazzeo and Rubinstein \cite{JMR}. We state our main analytic result that provides an analogue of classical parabolic Schauder theory in the singular setup. Finally we state our short-time existence result for the Yamabe flow precisely. Our subsequent work will deal with long-time behaviour and convergence of the flow.

\subsection{Simple edge spaces}

The subsequent discussion of simple edge spaces follows closely the presentation in \cite{MazVer:ATO}.
Consider a compact stratified pseudomanifold $\overline{M}$, with the top-dimensional open and dense stratum $M$ 
and a single lower-dimensional stratum $B$, which is a compact smooth manifold by the axioms on the differential and topological structure of stratified spaces.

The stratum $B$ admits a neighbourhood $\overline{\calU}\subset \overline{M}$ that is a fibration of truncated cones over $B$.
More specifically, there is a radial function $x:\overline{\calU}\to [0,1]$ and a submersion $\phi: \overline{\calU} =\overline{\calU}\cap M \to B$, 
such that the preimages $\phi^{-1}(y)\cap \calU$ are all diffeomorphic to open truncated cones $\mathscr{C}(F)=(0,1]_x\times F$ over a compact smooth manifold $F$. The restriction of $x$ to each fibre $\phi^{-1}(y)$ is a radial function on that cone. The level set $Y:=x^{-1}(1)$ is the total space of a fibration $\phi: Y \to B$ with fibre $F$, smooth and compact, and corresponds to the regular boundary of the neighbourhood $\calU$. Note that we do not distinguish 
notationally between the fibration $\phi:Y \to B$ with fibre $F$ and the corresponding 
fibration of cones $\phi:\overline{\calU} \to B$ with fibre $\overline{\mathscr{C}(F)}$.

We fix dimensions for the remainder of the paper.  We set $m = \dim M$, $b=\dim B$ and $f = \dim F$. Assume $f\geq 1$ and note that
\[ m = 1 + f + b. \]

We say that a Riemannian metric $g_0$ on $M$ is a {\it rigid} incomplete edge metric if it is smooth away from $\calU$, and has the following form over $\calU$: there exists a smooth Riemannian metric $\ha$ on $B$ and a symmetric $2$-tensor $\ka$ on $Y$, that restricts to Riemannian metrics on fibres $F$, such that 
\[
\left. g_0 \right|_{\calU} = dx^2 + \phi^*\ha +  x^2 \ka.
\]

Locally in a trivializing neighbourhood of the fibration $\phi$, this describes $\calU$ simply as the cone bundle over $B$ with open truncated cones $\mathscr{C}(F)=(0,1]\times F$ as fibres, with the metric induced by $g_0$ on each fibre $\phi^{-1}(y)$ being an exact warped-product conic metric. 

More generally, we call $g$ an incomplete edge metric if $g = g_0 + h$ where $g_0$ is rigid in the sense above and where $h$ is a smooth symmetric $2$-tensor such that $|h|_{g_0}$ is smooth on $[0,1)_x \times Y$ and vanishes at $x=0$. We refer to any compact stratified Riemannian space with the structure just described, and with an incomplete edge metric, as a simple edge space $(M,g)$ with edge data $(B,F,\phi)$. 

Among the incomplete edge metrics there is a slightly more restricted class of \emph{asymptotically admissible} edge metrics $g=g_0+h$ satisfying the following two conditions. First we
require $\phi: (Y, \left. g \right|_Y) \to (B,\ha)$ to be a Riemannian submersion with respect to $g_0$. Specifically, for $p\in Y$ the thangent space $T_pY$ splits into vertical and horizontal subspaces, $T^V_p Y \oplus
T^H_p Y$, where $T^V_pY$ is the tangent space to the fibre of $\phi$ through $p$ and $T^H_p Y$ is the corresponding orthogonal complement. The submersion condition requires that the tensor $\ka$ restricted to $T^H_p Y$ vanishes. 

We refer to an edge metric as \emph{asymptotically admissible} if additionally for every $y \in B$, the fibres $(F,\ka |_{\phi^{-1}(y)})$ are isospectral. Both admissibility conditions are satisfied in case of $\dim B=0$. 

These two conditions are fundamental for a microlocal heat kernel construction on edges by refining the parametrix to the normal operator of the Laplacian on $(M,g)$ 
in an iterative procedure, see \cite{MazVer:ATO}. We hint here at the origin of these 
two conditions. The first condition guarantees 
that the normal operator of the Laplacian splits into the Laplacian on the 
cone $(\mathscr{C}(F), ds^2 + s^2 \kappa)$ and the Euclidean Laplacian 
on $(\R^b, du^2)$. The second condition ensures polyhomogeneity of the resulting heat kernel, when lifted to the corresponding blowup space, see also the remark below. 

\begin{remark}
The condition on $(F,\ka |_{\phi^{-1}(y)}), y\in B$ to be isospectral ensures polyhomogeneity of the 
heat kernel when lifted to the corresponding parabolic blowup space. Our arguments in this paper however require the heat kernel
just to be conormal with only partial asymptotic expansions at the various boundary faces of the blowup. Thus while we only use polyhomogeneity of the heat kernel to a limited extent, we nevertheless pose the isospectrality condition for convenience, as the actual constructions of the various heat kernels in \cite{Moo:HKA} and \cite{MazVer:ATO}, as written down seem to require polyhomogeneity in an essential way.
An introduction into the various elements of blowups and polyhomogeneous distributions can be found in Grieser \cite{Gr}.
\end{remark}

The actual Schauder-type estimates for the Laplacian on $(M,g)$
require the heat kernel not only to admit a partial asymptotic expansion at the 
various boundary faces of the corresponding blowup space, but also requires specific structure of certain coefficients in the expansion. This leads to posing the following three conditions, which unlike asymptotic admissibility of $g$, summarize the analytic conditions special to deriving our estimates for the linear problem.
 
\begin{defn}\label{feasible}
Let $(M,g)$ be a simple edge space  with edge data $(B,F,\phi)$ 
and a singular neighbourhood $\calU\subset M$ of the edge $B$, such that 
the incomplete edge metric $g$ is asymptotically admissible and of the form $g=g_0+h$ with 
\[
\left. g_0 \right|_{\calU} = dx^2 + \phi^*\ha +  x^2 \ka.
\]
We call $(M,g)$ a feasible edge space, if the following three conditions are satisfied.
\begin{enumerate}
\item $h$ vanishes to second order at $x=0$, i.e. $|h|_{g_0}=O(x^2)$ as $x\to 0$.
\item the lowest non-zero eigenvalue $\lambda_0 > 0$ of Laplacians $\Delta_{\ka, y}$ associated to $(F,\ka |_{\phi^{-1}(y)})$ for any $y\in B$, satisfies
$\lambda_0 > \dim F$.
\item The Laplacian $\Delta_Y$ on $(Y,\phi^*\ha + \ka)$ satisfies the following condition: For any $u\in C^\infty(B)$, the function $\Delta_Y \phi^*u$ is the pullback of a function on the base $B$.
\end{enumerate}
\end{defn}

These three conditions yield refined heat kernel asymptotics in Proposition \ref{heat-expansion}, which in turn are crucial in the estimates below. We now discuss these hypotheses.  Given an asymptotically admissible edge space with edge data $(B,F,\phi)$, we indicate the dependence of $g$ on the metric structures on fibres $F$ and the base 
$B$ explicitly by writing $g=g(\ha,\ka)$.

The first condition is required to control higher order terms in the expansion \eqref{lap-exp} for the singular Laplacian.  The second condition is somewhat mild and can always be achieved by a rescaling of the fibre metric.  To see this, let $\lambda_0>0$ be the lowest non-zero eigenvalue of the (isospectral) Laplacians  $\Delta_{\ka, b}$ associated to $(F,\ka |_{\phi^{-1}(y)})$ for any $y\in B$.  For any $c>\sqrt{f/\lambda_0}$ the edge space $(M,g(\ha,c^{-2}\ka))$ satisfies Definition \ref{feasible} (ii).  In the special case of $\dim M=2$ and $F=\mathbb{S}^1$, the condition (ii) of Definition \ref{feasible} is satisfied whenever the singular neighbourhood $\calU\subset M$ of the edge $B$ is a fibration of cones of a fixed cone angle strictly less than $\pi/4$.

We can give a sufficient geometric condition for the third condition of Definition \ref{feasible}.  Recall the fibres of a Riemannian submersion admit a second fundamental form  (see for example \cite{Bes:EM}).  Working in an orthonormal frame for $Y$ that respects the splitting of vertical and horizontal subspaces, one finds condition (iii) holds if the mean curvature vector field of the fibres is projectable.  This holds for example in case the fibres are minimal.  This condition also appeared in J. M\"uller's study of Hodge theory of fibred cusp metrics \cite{Mul:ZKU}.

Finally, we should point out that partial differential operators modelled on spaces with isolated conical singularities were studied
by Cheeger \cite{Che:SGO}, Lesch \cite{Les:OOF}, Melrose \cite{Mel:TAP}, Br\"uning and Seeley \cite{BS} and Gil and Mendoza \cite{GM},
to name just a few.  Extensions to spaces with simple edge singularities, as studied here, were developed by Mazzeo
\cite{Maz:ETO}, and Schulze and his collaborators \cite{Sch1}, see also Br\"uning-Seeley \cite{BS3} and more recently 
Gil-Krainer-Mendoza \cite{GKM}. Further extensions to spaces with iterated edge singularities have been initiated in \cite{ALMP}.

\subsection{Mapping properties of the heat operator on H\"older spaces}

The proof of short-time existence for the Yamabe flow is based on a careful analysis of the mapping properties of 
the heat operator acting between H\"older spaces, defined with respect to an 
incomplete edge metric. Consider a feasible edge space $(M,g)$ as above.

\begin{defn}
The H\"older space $\ho, \A\in (0,1),$ is defined as the space of functions 
$u(p,t)$ that are continuous on $\overline{M} \times [0,T]$, such that their $\A$-th H\"older norm

\begin{align*}
\|u\|_{\A}:=\|u\|_{\infty} + \sup \left(\frac{|u(p,t)-u(p',t')|}{d_M(p,p')^{\A}+
|t-t'|^{\frac{\A}{2}}}\right) <\infty, 
\end{align*}
where $d_M(p,p')$ represents the distance between $p,p'\in M$ with respect to the 
incomplete edge metric $g$.  In local coordinate charts in the singular edge neighbourhood $\mathscr{U}$ may be
equivalently defined by
$$d_M((x,y,z), (x',y',z'))=\sqrt{|x-x'|^2+(x+x')^2|z-z'|^2 + |y-y'|^2}.$$
\end{defn}

\begin{defn}

Consider the Laplacian $\Delta_g$ of $(M,g)$ and the edge vector fields $\V$ 
introduced in \cite{Maz:ETO} and reviewed at the beginning of Section \ref{s-asymptotics}. Let $\mathcal{V}'_e$ be 
the space of edge vector fields with local set of generators $\{x\partial_x, x\partial_y, \partial_z\}$, where we 
require the $x\partial_y$ coefficient to be fibrewise constant at $x=0$, and put
$$\dom:=\{\partial_t, \Delta_g, x^{-1}V \mid \ V \in \mathcal{V}'_e\}.$$
Then the higher order H\"older spaces are defined as follows
\begin{equation*}
\begin{split}
\hhok = \{u\in \Lambda^{\A,\frac{\A}{2}} \mid &\Delta^j_g u, X\circ \Delta^j_g u \in \ho, \\
&X\in \dom, j=0,...,k-1\},
\end{split}
\end{equation*}
where differentiation is a priori in the distributional sense, and the H\"older norm 
$$\|u\|_{2k+\A}= \|u\|_{\A}+ \sum_{j=0}^{k-1} \|\Delta^j_g u\|_{\A} +
\sum_{j=0}^{k-1} \sum_{X\in \dom} \| X\circ \Delta^j_g u\|_{\A}.$$
\end{defn}

Though not classical, these \emph{hybrid} H\"older spaces seem not only to provide
a framework for treatment of parabolic Schauder-type estimates on incomplete edges, 
but also play an essential role in the analysis of K\"ahler-Einstein edge metrics,
see \cite{Donaldson, JMR}. 

Our first main result discusses the mapping properties of the heat operator on simple edge spaces, 
and essentially correspond to the classical parabolic Schauder estimates.

\begin{thm}
Let $(M^m,g)$ be a feasible edge space with edge data $(B^b,F^f,\phi)$ 
and a singular neighbourhood $\calU\subset M$ of the edge $B$.
The incomplete edge metric is given by $g=g_0+h$ with 
\[
\left. g_0 \right|_{\calU} = dx^2 + \phi^*\ha +  x^2 \ka.
\]
Let $\lambda_0>f$ be the smallest non-zero eigenvalue of $\Delta_{\kappa, y},y\in B$. Put 
\[
 \A_0 := -\frac{(f-1)}{2}+\sqrt{\frac{(f-1)^2}{4}+\lambda_0}-1>0.
\]
Denote by $\Delta_g$ the Friedrichs extension of the Laplacian on $(M,g)$. Then for $\A\in (0,
\min\{\A_0,1\})$ the heat operator $e^{-t\Delta_g}$ acts as a 
bounded convolution operator with respect to the H\"older spaces $(k\in \N_0)$
\begin{equation*}
\begin{split}
&e^{-t\Delta_g}: \hhok \to \hhokk,\\
&e^{-t\Delta_g}: \hhok \to \sqrt{t} \hhok, 
\end{split}
\end{equation*}
\end{thm}

As elaborated immediately after the definition of feasibility, the relation $\lambda_0>\dim F$ can 
be achieved using a rescaling of the metric $\ka$, and corresponds to 
choosing small cone angles in the surface case. This restriction guarantees 
$\A_0>0$. For $\dim M=2$ and the singular neighbourhood 
$\calU\subset M$ of the edge $B$ being a fibration of cones over $F=\mathbb{S}^1$ of a fixed cone angle $\theta <\pi/4$,
we have $\A_0=\textup{cotan} \, \theta -1>0$. In the rest of the paper $\A_0$ and $\A\in (0,\min\{\A_0,1\})$ are fixed.

We take this opportunity to mention a related paper.  In joint work with Emily Dryden \cite{BDV:THE}, the authors have also obtained mapping properties of the heat operator for the homogeneous Cauchy problem for the heat equation using time-weighted (spatial) H\"older spaces adapted to the incomplete metric.  The estimates we obtained there are more naturally applied to semilinear parabolic problems and are closely related to the work 
of Jeffres and Loya \cite{JefLoy:RSH}.  

\subsection{Yamabe flow on incomplete edge spaces}

We now return to the Yamabe flow given in equation \eqref{YF}.
Let $(M^m,\ginit)$ be a feasible edge space.  The scalar curvature of $\ginit$ becomes unbounded near the edge $x=0$, so further restrictions are required to ensure the flow exists.  
In addition to the feasibility of $\ginit$ we will require regularity of its scalar curvature
\begin{align}\label{hypo}
\scal(\ginit) \in \Lambda^{2k+\A}(M,\ginit),
\end{align}
for some $k\in \N$, where $\Lambda^{2k+\A}(M,\ginit)\subset \hhok$ is the subset of all time-independent functions of $\hhok$,
defined with respect to $\ginit$.  This already implies that the obstruction $\scal(\ka_{\textup{init}}) = f(f-1)$ is satisfied.

Returning to the Yamabe flow, since the flow preserves the conformal class of $\ginit$
we write
\[ g(t) =  e^{2 u(p,t)} \cdot \ginit, \]
which transforms the Yamabe flow to a scalar equation for $u$
\begin{equation*}
\partial_t u = - \frac{e^{-2u}}{2}  \left( 2(m-1) \Delta_{\ginit} u - (m-2)(m-1) |\nabla u|^2 + \scal(\ginit) \frac{}{}\right), 
\quad u(p,0) = 0,
\end{equation*}
where $\Delta_{\ginit}$ is the Laplacian of $(M,\ginit)$.
This is a quasilinear equation for $u$ and our second main result concerns 
local existence and uniqueness of its solutions.

\begin{thm} \label{mainthm}
Let $(M^m,\ginit)$ be a feasible edge space such that $\scal(\ginit) \in \Lambda^{2k+\A}(M,\ginit)$
for some $k\in \N$. Then the (transformed) Yamabe flow equation 
admits a solution $u \in \hhok$ for some $T>0$ and $\A\in (0,\A_0)$.
\end{thm}

The metric $e^{2u} \ginit$ is a simple edge metric that satisfies the last two conditions of feasibility after a suitable transformation of the defining function.  This transformation may only alter the order of decay of the nonrigid part of the metric so that $|h|_{g_0} = O(x)$.
Thus, Theorem \ref{mainthm} indeed asserts short-time existence of Yamabe flow within a class of compact Riemannian spaces with simple 
admissible edge singularities. 

The outline of the remainder of this paper is as follows.  In Section \ref{s-asymptotics} we discuss the asymptotics of the heat kernel of the Laplacian.  In Section \ref{schauder-est} we derive the analogue of Schauder estimates.  In Section \ref{sec:contraction} we prove a short-time existence result for certain quasilinear parabolic equations, and in Section \ref{sec:yamabe} we condition the Yamabe flow to apply the result of Section \ref{sec:contraction}.

Due to the length of this paper we decided to postpone a discussion of uniqueness, long-time existence and convergence results for the Yamabe flow for these metrics.  This will be done in a forthcoming manuscript.

\emph{Acknowledgments:} 
It is a pleasure for both authors to thank Rafe Mazzeo for his continuous support and many helpful discussions. 
The second author would also like to thank Matthias Lesch for valuable discussions on function 
analytic aspects of this work. Furthermore the authors would like to thank Simon Brendle for helpful remarks. 
The second author gratefully acknowledges financial support by the German Research Foundation DFG as well as 
the Hausdorff Institute in Bonn, and also thanks Stanford University for hospitality.

\section{Asymptotics of the heat kernel on simple edge spaces}\label{s-asymptotics}
Let $(M,g)$ be a simple edge space  with edge data $(B,F,\phi)$ 
and a singular neighbourhood $\calU\subset M$ of the edge $B$.
We pass from the singular space $\overline{M}$ to its resolution
$\widetilde{M}$, which is a manifold with boundary $Y=\partial \widetilde{M}$, obtained by replacing 
$\calU$, which is a bundle of cones, with the associated bundle of cylinders, where each fibre is now $[0,1)_x \times F$.
This resolution process is described in greater detail in \cite{Maz:ETO}.
The following summarizes classical geometric notions on singular edge manifolds,
see \cite{Maz:ETO}, also explained in \cite{MazVer:ATO} and repeated here for convenience of the reader. 

We describe adapted local coordinates $(x,y,z)$ on $\widetilde{M}$ near the boundary 
$Y=\partial \widetilde{M}$. Here $x$ is the radial function defining the boundary, $y$ is
the lift of a local coordinate system on $B^b$ and $z$ restricts to a 
local coordinate system on each fibre $F^f$. The class of edge
vector fields $\calV_e$ on $\widetilde{M}$ are those which are smooth and 
tangent to the fibres of $Y$ at $\partial \widetilde{M}$, given in the adapted coordinates by
a sum of smooth multiples of the generators $x\del_x$, $x\del_{y_i}$ 
and $\del_{z_j}$ with $i=1,...,b$ and $j=1,...,f$.

The edge vector fields define the class of differential edge operators 
$\textup{Diff}^*_e(M)$ acting on smooth functions over $M$. 
Any $L \in \textup{Diff}^*_e(M)$ is of the form
\[
L=\sum_{j+|\A|+|\beta|\leq n} a_{j,\A,\beta}(x,y,z)(x\partial_x)^j(x\partial_y)^{\A}\partial_z^{\beta},
\]
with each $a_{j,\alpha,\beta}$ smooth. 
The operator is called edge elliptic if its edge symbol
\[
{}^e \sigma_n(L)(x,y,z;\xi,\eta,\zeta) :=  \sum_{j + |\A| + |\beta| = n} a_{j,\A,\beta}(x,y,z) \xi^j \eta^{\A} \zeta^{\beta}
\]
is invertible for $(\xi,\eta,\zeta) \neq (0,0,0)$. 
The edge symbol also admits a global definition as a
function on a certain edge cotangent bundle ${}^e T^* \widetilde{M}$.
The asymptotic structure of solutions to an elliptic differential edge operator $L$ 
is encoded in its indicial operator, which acts on functions on $\RR^+ \times F$, for any fixed point $y_0\in B$
\[
I(L)_{y_0} := \sum_{j + |\beta| \leq m} a_{j,0,\beta}(0,y_0,z) (s\partial_s)^j \partial_z^\beta.
\]
This is equivalent, by taking the Mellin transform in the $s$ $(\in \RR^+$) variable, to a family of holomorphic operators on $F$, 
\[
I_\zeta(L)_{y_0} = \sum_{j + |\beta| \leq m} a_{j,0,\beta} (0,y_0,z) \zeta^j \partial_z^\beta.
\]
Values of $\zeta\in \C$ for which $I_\zeta(L)_{y_0}$ is not invertible (acting on $L^2(F)$) 
are called indicial roots of $L$; they are analogous
to eigenvalues for this operator family on the compact manifold $F$. 

As we now explain, for the scalar Laplacian $\Delta$ of a feasible edge space $(M,g)$,
$x^2\Delta \in \textup{Diff}^2_e(M)$, where we omit the subindex $g$ when there is no 
danger of confusion. Recall that the fibration $\phi:Y\to B$ is a 
Riemannian submersion with respect to the rigid part $g_0$ of the edge metric 
$g=g_0+h$, where $g_0\restriction \mathscr{U} = dx^2 + \phi^*\ha + x^2\ka$ and $|h|_{g_0}=O(x^2)$ 
as $x\to 0$. Consider the orthogonal splitting of $TY$ into vertical and horizontal subbundles
$$TY=\ker d\phi \oplus (\ker d\phi)^{\perp},$$
where $d\phi:(\ker d\phi)^{\perp} \to B$,
is an isomorphism and the horizontal subbundle $(\ker d\phi)^{\perp}$ is annihilated by $\ka$.
Under the identification $(\ker d\phi)^{\perp} \cong B$ the Laplacian $\Delta$ 
is given over the singular neighbourhood $\mathscr{U}$ by the following expression 
\begin{align}\label{lap-exp}
x^2\Delta\restriction \mathscr{U}=-(x\partial_x)^2 - (\dim F - 1)x \partial_x + \Delta_{F} + x^2\Delta_B +
x^2\mathscr{O}, \quad x^2\mathscr{O}\in x\textup{Diff}^2_e(M),
\end{align}
where $\Delta_F$ is the pullback of the Laplacian of $(F,\ka|_{\phi^{-1}(y)})$ 
to fibres of the fibration $\phi:Y\to B$, and $\Delta_B$ is the lift of 
the Laplacian for $(B,\ha)$. The higher order term $\mathscr{O}=\mathscr{O}_h+\mathscr{O}_\phi$ is comprised 
of the contributions $\mathscr{O}_h$ from $h$ as well as the contributions $\mathscr{O}_\phi$ from the 
second fundamental form and the curvature of the fibration $\phi$.  
Note that by the first condition of feasibility in Definition \ref{feasible}, we have $\mathscr{O}_h \in \textup{Diff}^2_e(M)$.  

The Mellin transform of the corresponding 
indicial operator is given by the holomorphic operator family on $F$ 
\[
I_\zeta(x^2\Delta)_{y_0}= -\zeta^2 - (\dim F - 1)\zeta + \Delta_{\ka,y_0}.
\]

The tangential operators $\Delta_{\ka,y_0}$ on $(F,\ka|_{\phi^{-1}(y_0)}), y_0\in B$, are spectrally equivalent 
and hence indicial roots of $x^2\Delta$ are $y_0$-independent. 

Denote the Friedrichs extension of the Laplacian on $(M,g)$ again by $\Delta$. 
We denote its heat operator by $e^{-t\Delta}$ and the corresponding heat kernel by $H$. 
The heat operator of $\Delta$ acts as an integral convolution operator on $u(t,\cdot)\in \dom (\Delta), t>0,$ 
\begin{equation} \label{eqn:hk-on-functions}
e^{-t\Delta}*u(t,p) = \int_0^t \int_M H\left( t-\wt, p,\widetilde{p} \right) u(\wt, \widetilde{p}) \dv (\widetilde{p}) \, d\wt,
\end{equation}
and solves the inhomogeneous heat problem
\begin{equation*}
(\partial_t + \Delta) w(t,p)  = u(t,p), \ w(0,p)=0,
\end{equation*}
for any $u(t,\cdot)\in \dom (\Delta), t>0$. The heat kernel can be viewed a priori 
as a distribution on $M^2_h=\R^+\times \widetilde{M}^2$, 
with the local coordinates near the corner in $M^2_h$ given by $(t, (x,y,z), (\widetilde{x}, \wy, \widetilde{z}))$, 
where $(x,y,z)$ and $(\widetilde{x}, \wy, \widetilde{z})$ are the coordinates on the two copies of $M$ near the singularity. 
The kernel $H(t, (x,y,z), (\wx,\wy,\wz))$ has non-uniform behaviour at the submanifolds
\begin{align*}
&A =\{ (t, (x,y,z), (\wx,\wy,\wz))\in M^2_h \mid t=0, \, x=\wx=0, \, y= \wy\}, \\
&D =\{ (t, p, \widetilde{p})\in M^2_h \mid t=0, \, p=\widetilde{p}\},
\end{align*}
which requires an appropriate blowup of the heat space $M^2_h$, 
such that the corresponding heat kernel lifts to a polyhomogeneous distribution 
in the sense of the following definition, also cited from \cite{Mel:TAP} in various other related publications, 
cf. \cite{MazVer:ATO}.

\begin{defn}\label{phg}
Let $\mathfrak{W}$ be a manifold with corners, with $\{(H_i,\rho_i)\}_{i=1}^N$ being an enumeration 
of its (embedded) boundaries and the corresponding defining functions. For any multi-index $b= (b_1,
\ldots, b_N)\in \C^N$ we write $\rho^b = \rho_1^{b_1} \ldots \rho_N^{b_N}$.  Denote by $\mathcal{V}_b(\mathfrak{W})$ the space
of smooth vector fields on $\mathfrak{W}$ which lie
tangent to all boundary faces. A distribution $\w$ on $\mathfrak{W}$ is said to be conormal,
if $\w$ is a restriction of a distribution across the boundary faces of $\mathfrak{W}$, 
$\w\in \rho^b L^\infty(\mathfrak{W})$ for some $b\in \C^N$ and $V_1 \ldots V_\ell \w \in \rho^b L^\infty(\mathfrak{W})$
for all $V_j \in \mathcal{V}_b(\mathfrak{W})$ and for every $\ell \geq 0$. An index set 
$E_i = \{(\gamma,p)\} \subset {\mathbb C} \times {\mathbb N_0}$ 
satisfies the following hypotheses:

\begin{enumerate}
\item $\textup{Re}(\gamma)$ accumulates only at plus infinity,
\item For each $\gamma$ there is $P_{\gamma}\in \N_0$, such 
that $(\gamma,p)\in E_i$ iff $p \leq P_\gamma$,
\item If $(\gamma,p) \in E_i$, then $(\gamma+j,p') \in E_i$ for all $j \in {\mathbb N_0}$ and $0 \leq p' \leq p$. 
\end{enumerate}
An index family $E = (E_1, \ldots, E_N)$ is an $N$-tuple of index sets. 
Finally, we say that a conormal distribution $\w$ is polyhomogeneous on $\mathfrak{W}$ 
with index family $E$, we write $\w\in \mathscr{A}_{\textup{phg}}^E(\mathfrak{W})$, 
if $\w$ is conormal and if in addition, near each $H_i$, 
\[
\w \sim \sum_{(\gamma,p) \in E_i} a_{\gamma,p} \rho_i^{\gamma} (\log \rho_i)^p, \ 
\textup{as} \ \rho_i\to 0,
\]
with coefficients $a_{\gamma,p}$ conormal on $H_i$, polyhomogeneous with index $E_j$
at any $H_i\cap H_j$. 
\end{defn}

The basic underlying idea of the heat-space blowup is to capture the well-known scaling 
property of the cone and the resulting polyhomogeneous behaviour of the heat kernel 
as its entries approach certain submanifolds of $M^2_h$ from the various angles. 
This is done by ``blowing" up these submanifolds, parabolically in time-direction, 
a procedure introduced by Melrose, see \cite{Mel:TAP} and in this particular case \cite{MazVer:ATO}. 

The blowup procedure corresponds to introducing certain (polar) coordinates on $M^2_h$, 
such that the heat kernel admits polyhomogeneous asymptotic expansions at $A$ and $D$ 
in these coordinates, together with a unique minimal differential structure with respect to 
which these coordinates are smooth. To get the correct blowup of $M^2_h$ we first parabolically 
(i.e. we treat $\sqrt{t}$ as a smooth variable) blow up the submanifold $A$.

The resulting heat-space $[M^2_h, A]$ is defined as the union of
$M^2_h\backslash A$ with the interior spherical normal bundle of $A$ in $M^2_h$. 
The blowup $[M^2_h, A]$ is endowed with the unique minimal differential structure 
with respect to which smooth functions in the interior of $M^2_h$ and polar coordinates 
on $M^2_h$ around $A$ are smooth. This blowup introduces a new boundary 
hypersurface $-$ the front face ff; the boundary faces $\{x=0\}, \{\wx=0\}$ and $\{t=0\}$
lift to rf (the right face), lf (the left face) and tf (the temporal face), respectively.  

The actual heat-space blowup $\mathscr{M}^2_h$ is obtained by a parabolic blowup 
of $[M^2_h, A]$ along the diagonal $D$ lifted to a submanifold of $[M^2_h, A]$. The resulting blowup $\mathscr{M}^2_h$ is 
defined as before by cutting out the submanifold and replacing it with its spherical 
normal bundle, which constitutes a new boundary face $-$ the temporal diagonal td. 
It is a manifold with boundaries and corners, visualized in Figure 1.

\begin{figure}[h]
\begin{center}
\begin{tikzpicture}
\draw (0,0.7) -- (0,2);
\draw[dotted] (-0.1,0.7) -- (-0.1, 2.2);
\node at (-0.4,2) {t};

\draw(-0.7,-0.5) -- (-2,-1);
\draw[dotted] (-0.69,-0.38) -- (-2.05, -0.9);
\node at (-2.05, -0.6) {$x$};

\draw (0.7,-0.5) -- (2,-1);
\draw[dotted] (0.69,-0.38) -- (2.05, -0.9);
\node at (2.05, -0.6) {$\wx$};

\draw (0,0.7) .. controls (-0.5,0.6) and (-0.7,0) .. (-0.7,-0.5);
\draw (0,0.7) .. controls (0.5,0.6) and (0.7,0) .. (0.7,-0.5);
\draw (-0.7,-0.5) .. controls (-0.5,-0.6) and (-0.4,-0.7) .. (-0.3,-0.7);
\draw (0.7,-0.5) .. controls (0.5,-0.6) and (0.4,-0.7) .. (0.3,-0.7);

\draw (-0.3,-0.7) .. controls (-0.3,-0.3) and (0.3,-0.3) .. (0.3,-0.7);
\draw (-0.3,-1.4) .. controls (-0.3,-1) and (0.3,-1) .. (0.3,-1.4);

\draw (0.3,-0.7) -- (0.3,-1.4);
\draw (-0.3,-0.7) -- (-0.3,-1.4);

\node at (1.2,0.7) {\large{rf}};
\node at (-1.2,0.7) {\large{lf}};
\node at (1.1, -1.2) {\large{tf}};
\node at (-1.1, -1.2) {\large{tf}};
\node at (0, -1.7) {\large{td}};
\node at (0,0.1) {\large{ff}};
\end{tikzpicture}
\end{center}
\label{heat-incomplete}
\caption{Heat-space Blowup $\mathscr{M}^2_h$.}
\end{figure}

The projective coordinates on $\mathscr{M}^2_h$ are then given as follows. 
Near the top corner of the front face ff, the projective coordinates are given by
\begin{align}\label{top-coord}
\rho=\sqrt{t}, \  \xi=\frac{x}{\rho}, \ \widetilde{\xi}=\frac{\wx}{\rho}, \ u=\frac{y-\wy}{\rho}, \ z, \ \wy, \ \wz,
\end{align}
where in these coordinates $\rho, \xi, \widetilde{\xi}$ are the defining 
functions of the boundary faces ff, rf and lf respectively. 
For the bottom corner of the front face near the right hand side projective coordinates are given by
\begin{align}\label{right-coord}
\tau=\frac{t}{\wx^2}, \ s=\frac{x}{\wx}, \ u=\frac{y-\wy}{\wx}, \ z, \ \wx, \ \wy, \ \widetilde{z},
\end{align}
where in these coordinates $\tau, s, \widetilde{x}$ are
the defining functions of tf, rf and ff respectively. 
For the bottom corner of the front face near the left hand side
projective coordinates are obtained by interchanging 
the roles of $x$ and $\widetilde{x}$. Projective coordinates 
on $\mathscr{M}^2_h$ near temporal diagonal are given by 
\begin{align}\label{d-coord}
\eta=\frac{\sqrt{t}}{\wx}, \ S =\frac{(x-\wx)}{\sqrt{t}}, \ 
U= \frac{y-\wy}{\sqrt{t}}, \ Z =\frac{\wx (z-\wz)}{\sqrt{t}}, \  \wx, \ 
\wy, \ \widetilde{z}.
\end{align}

In these coordinates tf is the face in the limit $|(S, U, Z)|\to \infty$, 
ff and td are defined by $\widetilde{x}, \eta$, respectively. 
The blow-down map $\beta: \mathscr{M}^2_h\to M^2_h$ is in 
local coordinates simply the coordinate change back to 
$(t, (x,y, z), (\widetilde{x},\wy, \widetilde{z}))$. The heat kernel 
lifts to a polyhomogeneous conormal distribution on $\mathscr{M}^2_h$
with product type expansions at the corners of the heat space, and 
with coefficients depending smoothly on the tangential variables. 
More precisely we have for asymptotically admissible edge metrics
the following

\begin{thm}\label{friedrichs-blowup}\textup{(\cite[Theorem 1.2]{MazVer:ATO})}
Let $(M^m,g)$ be a simple edge space with an asymptotically admissible edge metric $g$. Let the 
heat kernel associated to the Friedrichs extension of the corresponding Laplacian be denoted by $H$.
The lift $\beta^*H$ is a polyhomogeneous conormal distribution on $\mathscr{M}^2_h$ 
of leading order $(-\dim M)$ at the front face ff, leading order $(-\dim M)$ at the temporal diagonal td, and 
continuous at the left and right boundary faces. $\beta^*H$ vanishes to infinite order 
at the temporal face tf and does not admit logarithmic terms in its asymptotic expansion at ff and td. 
\end{thm} 

The heat kernel asymptotics were derived in \cite[Theorem 1.2]{MazVer:ATO} under a unitary rescaling 
transformation in the singular neighbourhood $\mathscr{U}=(0,1)$ of the edge singularity, 
similar to Br\"uning-Seeley \cite{BruSee:AIT}
$$\Phi:C_{0}^{\infty}(\mathscr{U})\to C_{0}^{\infty}(\mathscr{U}), \quad u\mapsto x^{f/2}u,$$
such that the transformed Laplacian $\Phi \circ \Delta \circ \Phi^{-1}$ is self-adjoint in 
$L^2\left(\mathscr{U}, dx\wedge \textup{vol}(\phi^*\ha + \ka)\right)$, where we omit factors in the volume
which are bounded and smooth up to $x=0$. In particular the leading order of the heat kernel asymptotics 
in \cite{MazVer:ATO} refer to the heat operator as an integral operator with respect to the volume 
$dx\wedge \textup{vol}\left(\phi^*\ha + \ka\right)$.  This accounts for the difference in powers of the defining 
functions between Theorem \ref{friedrichs-blowup} and \cite[Theorem 1.2]{MazVer:ATO}. 

To see this explicitly, we compare the actions of the heat operators with and without the unitary rescaling transformation. 
We may assume for simplicity that the heat kernel is compactly supported in $\mathscr{U}^2\times \R^+$  and 
for simplicity omitting factors in the volume which are bounded and smooth up to $x=0$, we find
\begin{align*}
\exp(-t \Phi \circ \Delta \circ \Phi^{-1})u=\int_{\mathscr{U}} 
\exp(-t\Phi \circ \Delta \circ \Phi^{-1})(x,y,z,\wx,\wy,\wz)u(\wx,\wy,\wz)  d\wx d\wy d\wz \\
= \int_{\mathscr{U}} x^{f/2} \exp(-t\Delta)(x,y,z,\wx,\wy,\wz) \wx^{-f/2}u(\wx,\wy,\wz) \wx^f  d\wx d\wy d\wz.
\end{align*}
Consequently we deduce the following relation between the heat kernels:
\begin{align*}
\exp(-t\Phi \circ \Delta \circ \Phi^{-1})(x,y,z,\wx,\wy,\wz) = (x\wx)^{f/2} \exp(-t\Delta)(x,y,z,\wx,\wy,\wz).
\end{align*}
One may now check in projective blowup coordinates that the lifts of the heat kernels to the blowup heat space are related as follows
\begin{align}\label{rescaling:heat}
\beta^*\exp(-t\Phi \circ \Delta \circ \Phi^{-1}) = \rho_{\textup{ff}}^f  (\rho_{\textup{rf}} \rho_{\textup{lf}})^{f/2}\beta^*\exp(-t\Delta),
\end{align}
where, e.g.,  $\rho_{\textup{rf}}$ is a defining function for the right face. The rescaled heat kernel $\exp(-t\Phi \circ \Delta \circ \Phi^{-1})$ has been discussed in \cite[Theorem 1.2]{MazVer:ATO} and in case of functions lifts to a density on the heat-space blowup with asymptotic expansion of leading order $(-1-b)$ at ff, of leading order $(f/2)$ at lf and rf, and of leading order $(-\dim M)$ at td. Employing the relation \eqref{rescaling:heat} we arrive at the actual statement of Theorem \ref{friedrichs-blowup}.

\begin{remark}
A similar situation arises in the special case of a cone, where the edge $B$ is collapsed to a point. This setup was considered by Mooers in \cite{Moo:HKA}, where the heat kernel was considered under a unitary rescaling transformation as well. Without the rescaling one needs to take the singular factor $x^f$ of the volume $x^f dx\wedge \textup{vol}g^{\partial M}(x)$ into account, leading to different asymptotic behaviour. 
\end{remark}

In fact, feasibility of the edge metric guarantees a finer result on 
the asymptotic behaviour of $H$ at the left and right boundary faces of $\mathscr{M}^2_h$.

\begin{prop}\label{heat-expansion}
Let $(M^m,g)$ be a feasible edge space  with edge data $(B^b,F^f,\phi)$. 
Then the lift $\beta^*H$ is a polyhomogeneous distribution on $\mathscr{M}^2_h$ with the
index set at rf given in terms of
$$
E=\{\gamma \geq 0 \mid  \gamma= -\frac{(f-1)}{2}+ \sqrt{\frac{(f-1)^2}{4}+ \lambda},
 \, \lambda \in \textup{Spec}\, \Delta_F\}.
$$
More precisely, $\beta^*H$ separates into two components $\beta^*H'$ and $\beta^*H''$ of leading order 
behaviour $(-m)$ and $(-m+1)$ at the front face, respectively; and
if $s$ is a defining function of the right face 
of the form \eqref{top-coord} or \eqref{right-coord} and 
$E^*=E\backslash\{0\}$, then the expansion of $\beta^*H'$ and $\beta^*H''$ takes the following form
\begin{equation*}
 \begin{split}
 &\beta^*H' \sim \sum_{\gamma \in E} s^{\gamma} a_{\gamma}(\beta^*H), \quad s \to 0, \\
 &\beta^*H'' \sim \sum_{l=2}^{\infty} s^la_l(\beta^*H)+
\sum_{\gamma \in E^*} \sum_{l=1}^{\infty}s^{\gamma+l} a_{\gamma,l}(\beta^*H),
\quad s \to 0,
 \end{split}
\end{equation*}
where $a_{0}(\beta^*H), a_{2}(\beta^*H)$ are constant in $z$.
\end{prop}

\begin{proof}
Recall the heat kernel construction in \cite{MazVer:ATO}. 
The initial approximate parametrix for the solution operator of $\mathcal{L}=(\partial_t+\Delta)$ 
is constructed in \cite{MazVer:ATO} by solving the heat equation to 
first order at the front face ff of $\mathscr{M}^2_h$. 
The restriction of the lift $\beta^*(t\mathcal{L})$ to ff is called the normal operator 
$N_{\ff}(t\mathcal{L})$ at the front face and is given in projective coordinates \eqref{right-coord} 
explicitly as follows
\[
N_{\ff}(t\mathcal{L})=\tau (\partial_{\tau} - \partial_s^2 - fs^{-1}\partial_s +s^{-2}\Delta_{F,z} + \Delta^{\R^b}_u) \\
=:\tau (\partial_{\tau} + \Delta^{\mathscr{C}(F)}_s + \Delta^{\R^b}_u).
\]

$N_{\ff}(t\mathcal{L})$ does not involve derivatives with respect to $(\wx, \wy,\wz)$ and hence acts 
tangentially to the fibres of the front face. Searching for an initial parametrix $H_0$, we solve 
the heat equation to first order at the front face, and note that 
\[
N_{\ff}(t\mathcal{L}\circ H_0)= N_{\ff}(t\mathcal{L}) \circ N_{\ff}(H_0) \equiv 0
\]
is the heat equation on the model edge $\mathscr{C}(F)\times \R^{b}_w$, 
where $\mathscr{C}(F)=\R^+_s\times F_z$ denotes the model cone. 
Consequently, the initial parametrix $H_0$ is defined by choosing $N_{\ff}(H_0)$ to 
equal the fundamental solution for the heat operator $N_{\ff}(t\calL)$, and extending 
$N_{\ff}(H_0)$ trivially to a neighbourhood of the front face.
Using the projective coordinates $(\tau,s,u,z,\wx,\wy,\wz)$ near $\ff$, see 
\eqref{right-coord}, we have
\begin{equation}\label{normal-heat-ops}
N_{\ff}(H_0) := H^{\mathscr{C}(F)}(\tau,s,z,1,\wz) H^{\RR^b}(\tau, u ,0),
\end{equation}
where $H^{\R^b}$ denotes the Euclidean heat kernel on $\R^b$, and $H^{\mathscr{C}(F)}$ is the heat kernel 
for the exact cone $\mathscr{C}(F)$, as studied by \cite{Che:SGO}, \cite{Les:OOF} and \cite{Moo:HKA}. 
Consequently, the index set $E_0$ for the asymptotic behaviour of $H_0$ at the right and left boundary faces 
is given in terms of the indicial roots $\gamma$ for the Laplacian $\Delta^{\mathscr{C}(F)}_s$ on the exact cone, i.e.
\begin{align*}
H_0 \sim \sum_{\gamma \in E} s^{\gamma} a_{\gamma}(H_0), \ s\to 0,
\end{align*}
with the leading coefficient $a_0(H_0)$ being harmonic on fibres and hence constant in $z$. 
Note that by condition (ii) of Definition \ref{feasible}, any $\gamma\neq 0$ is automatically $\gamma \geq 1+\A_0$.
The error of the initial parametrix $H_0$ is given by 
\begin{align*}
\beta^*(t\mathcal{L})H_0=\left(\tau \wx^2 \beta^* \mathscr{O}_h + 
\left[\beta^*(t\Delta_{g_0}) - \tau \Delta^{\mathscr{C}(F)}_s - \tau \Delta^{\R^b}_u\right]
\right) H_0=:P_0.
\end{align*}
Note that 
$\left[\beta^*(t\Delta_{g_0}) - \tau \Delta^{\mathscr{C}(F)}_s - \tau \Delta^{\R^b}_u\right]$ is 
of higher order in $s$ and is of higher order in $\wx$, since its normal operator at the front face is zero by construction. 
Using the fact that $a_0(H_0)$ is fibrewise constant, we find that the leading term 
in the $H_0$-expansion at $\rf$ is annihilated by $s\partial_s$ and $\partial_z$ (but not by $\partial_u$). 
Hence
\begin{align*}
\left[\beta^*(t\Delta_{g_0}) - \tau \Delta^{\mathscr{C}(F)}_s - \tau \Delta^{\R^b}_u\right] H_0 
&\sim \wx^{-m+1}\left(\sum_{l=0}^{\infty}s^l c_{l} +\sum_{\gamma \in E^*} \sum_{l=0}^{\infty}
s^{\gamma-1+l} c_{\gamma,l}\right),  \ s\to 0, \\
\left[\tau \wx^2 \beta^* \mathscr{O}_h \right] H_0 &\sim \wx^{-m+2}\left(\sum_{l=1}^{\infty}s^l d_{l} +\sum_{\gamma \in E^*} \sum_{l=0}^{\infty}
d^{\gamma+l} a_{\gamma,l}\right),  \ s\to 0,
\end{align*}
where the condition (iii) of Definition \ref{feasible} implies that $c_0$ is constant in $z$.
Therefore $P_0$ is of leading order $(-m+1)$ at the front face and 
its expansion at rf is given by 
\begin{align*}
P_0 \sim \sum_{l=0}^{\infty}s^l a_{l}(P_0) + 
\sum_{\gamma \in E^*} \sum_{l=0}^{\infty}
s^{\gamma-1+l} a_{\gamma,l}(P_0),  \ s\to 0,
\end{align*}
where $a_0(P_0)$ is fibrewise constant. 
The next step in the construction of the heat kernel involves adding a kernel $H_0'$ to $H_0$, 
such that the new error term is vanishing to infinite order at rf. 
In order to eliminate the term $s^\gamma a_\gamma$ in the asymptotic 
expansion of $P_0$ at rf, we only need to solve 
\begin{align}\label{indicial}
(-\partial_s^2- fs^{-1}\partial_s + s^{-2}\Delta_{F,z}) u = s^\gamma (\tau^{-1}a_\gamma).
\end{align}

This is because all other terms in the expansion of $t\mathcal{L}$ at rf lower the 
exponent in $s$ by at most one, while the indicial part lowers the exponent by two. 
The variables $(\tau, \w, \wx, \wy, \wz)$ enter the equation only as parameters. 
The equation on the model cone is solved via the Mellin transform and the solution is polyhomogeneous in 
all variables, including parameters and is of leading order $(\gamma+2)$. 
Consequently, the correcting kernel $H_0'$ must be of leading order $2$ at rf 
and of leading order $(-m+1)$ at ff, since $P_0$ is of order $(-m+1)$ at ff and
the defining function $\wx$ of the front face enters \eqref{indicial} only as a parameter. 

In case of $\gamma=0$, the indicial equation \eqref{indicial} reduces to 
$$
(-\partial_s^2- fs^{-1}\partial_s) u = s^0 (\tau^{-1}a_0),
$$
with $a_0$ and the solution $u= - s^2a_0(2+2f)^{-1}$ both independent of $z$. 
Hence, the leading order term in the expansion of $H_0'$ at rf is constant in $z$. 
Put $H_1:=H_0+H_0'$, where $H_0'$ is of leading order $(-m+1)$ at ff and expands near rf as follows
$$
H_0' \sim \sum_{l=2}^{\infty}s^la_l(H_0') + 
\sum_{\gamma \in E^*} \sum_{l=1}^{\infty}
s^{\gamma+l} a_{\gamma,l}(H_0'),  \ s\to 0,
$$
where we have explained that $a_{2}(H_0')$ is constant in $z$. The error 
of our new heat parametrix $H_1$ is now vanishing to infinite order at rf 
and is of leading order $(-m+1)$ at ff.

In the following correction steps the exact heat kernel is 
obtained from $H_1$ first by improving the error near td and then by an iterative correction procedure, 
adding terms of the form $H_1\circ (P_1)^k$, where $P_1:=t\mathcal{L}H_1$ 
is vanishing to infinite order at rf and td. 

Extend $s^\gamma a_\gamma(H_1)$ to a polyhomogeneous 
distribution on $\mathscr{M}^2_h$ with same index set at ff and lf as $H_1$, vanishing to 
infinite order at tf and td, and a single term $s^\gamma a_\gamma(H_1)$ in its expansion at rf. Then 
$s^\gamma a_\gamma(H_1)\circ (P_1)^k$ also makes sense, and since 
$\beta^*(x\partial_x - \gamma)s^\gamma a_\gamma(H_1)$ 
vanishes to infinite order at rf, so
$$
\beta^*(x\partial_x - \gamma) [s^\gamma a_\gamma(H_1)\circ (P_1)^k] \sim O(s^{\infty}), \ 
s\to 0.
$$ 

Consequently $s^\gamma a_\gamma(H_1)\circ (P_1)^k$ has a single $s^\gamma$
term in its expansion at the right face. Similarly, applying $\partial_z$, 
we show that the coefficients in the expansion of $s^0 a_0(H_1)\circ (P_1)^k$ and $s^2 a_2(H_1)\circ (P_1)^k$ 
at the rf are constant in $z$. We deduce that $\beta^*H=\beta^*H'+\beta^*H''$, with 
$\beta^*H'$ of leading order $(-m)$ at ff, $\beta^*H''$ of leading order $(-m+1)$ at ff, and
\begin{equation*}
 \begin{split}
 &\beta^*H' \sim \sum_{\gamma \in E} s^{\gamma} a_{\gamma}(\beta^*H), \quad s \to 0, \\
 &\beta^*H'' \sim \sum_{l=2}^{\infty} s^la_l(\beta^*H)+
\sum_{\gamma \in E^*} \sum_{l=1}^{\infty}s^{\gamma+l} a_{\gamma,l}(\beta^*H),
\quad s \to 0,
 \end{split}
\end{equation*}
where $a_{0}(\beta^*H), a_{2}(\beta^*H)$ are constant in $z$.
\end{proof}

In contrast to a similar statement in our previous work \cite{BDV:THE},
the main novelty here is the additional information on the higher order term
$a_2(\beta^*H)$. Its fibrewise constancy is the fundamental key
to the estimate of the $I_3$-integral at $\rf$ of $\mathscr{M}^2_h$ below.
The asymptotic expansions of $H'$ and $H''$ straightforwardly lead to the following

\begin{cor}\label{corr}
Consider a feasible edge space $(M,g)$ with $g=g_0+h$. Denote the heat kernel of the 
Friedrichs extension of the corresponding scalar Laplacian by $H$.
Write $$\Delta_{g_0} =: x^{-2}I(x^2\Delta_g) + L \circ \partial_y,$$
where the second summand is either a combination of $\{\partial_y, \partial_y^2\}$ coming from
$\ha$, or $\{\partial_z\partial_y\}$ coming from the curvature of the fibration. 
Let $\mathcal{V}'_e$ be the space of edge vector fields with local set of generators $\{x\partial_x, x\partial_y, \partial_z\}$, where we 
require the $x\partial_y$ coefficient to be fibrewise constant at $x=0$, and consider $X\in x^{-1}\mathcal{V}'_e$. Then
\begin{align*}
\beta^*(x^{-2}I(x^2\Delta_g)H) &\sim \rho_\ff^{-m-2} \rho_\td^{-m-2} \rho_\tf^\infty  \rho_\rf^\infty  +
\rho_\ff^{-m-1} \rho_\td^{-m-2} \rho_\tf^\infty (c_0 \rho_\rf^0 + O(\rho_\rf^{\A'_0})), \\
\beta^*(L H) &\sim \rho_\ff^{-m-1} \rho_\td^{-m-1} \rho_\tf^\infty (c_1 \rho_\rf^0 + O(\rho_\rf^{\A'_0})), \\
\beta^*(X H) &\sim \rho_\ff^{-m-1} \rho_\td^{-m-1} \rho_\tf^\infty (c_2 \rho_\rf^0 + O(\rho_\rf^{\A'_0})),
\end{align*}
where $c_{0,1,2}$ are fibrewise contant, i.e. independent of $z$; and $\A'_0:=\min\{\A_0,1\}$.
\end{cor}

\section{Parabolic Schauder-type estimates} \label{schauder-est}

The parabolic Schauder estimates here refer to boundedness property of the heat operator 
with respect to H\"older spaces, defined in the introductory Section \ref{intro}. 
The next auxiliary proposition identifies the H\"older space $\hho$ as a Banach space.

\begin{prop}\label{closed}
Let $X$ be a compact manifold with boundary, $C_0^{\infty}(X)$ smooth 
functions on $X$ with compact support in the interior and $\dom' (X)$ the 
space of distributions. Assume $H\subset \dom' (X)$ is a Banach space 
with norm $\| \cdot \|_H$ such that $C_0^{\infty}(X)\subset H$.
Let $D=(D_1,...,D_r)$ be a collection of 
differential operators on $X$, acting on elements of $H$ in the distributional 
sense.  Define
\begin{align*}
H_D:=\{u\in H \mid D_1 u, ... , D_r u\in H \}, \\
\|u\|_D:= \|u\|_H + \sum_{j=1}^r \|D_j u\|_H, \ u\in H.
\end{align*}
Then $(H_D,\| \cdot \|_D)$ is a Banach space.
\end{prop}
\begin{proof}
Let $(u_n)_{n\in \N} \subset H_D$ be a Cauchy sequence with respect to $\| \cdot \|_D$.
By completeness of $(H,\| \cdot \|_H)$, the sequences $(u_n)_{n\in \N}$ and 
$(D_j u_n)_{n\in \N}, j=1,...,r$ converge to $u\in H$ and $v_j\in H,j=1,...,r$, respectively. 
For any $\phi \in C_0^{\infty}(X)$ and $j=1,...,r$ we compute
\begin{align*}
(D_ju, \phi):=(u,D^t_j\phi)=\underset{n\to \infty}{\lim} (u_n,D^t_j \phi) 
=\underset{n\to \infty}{\lim} (D_j u_n,\phi) = (v_j,\phi).
\end{align*}
Hence, $D_ju=v_j\in H$ in the distributional sense. Consequently, $u\in H_D$ 
and therefore $(H_D,\| \cdot \|_D)$ is a Banach space. 
\end{proof}

Since the H\"older space $\ho$ and any collection 
of differential operators satisfies the conditions of Proposition \ref{closed}, 
the associated spaces $\hhok, k\in \N$ 
are Banach spaces as well.  Observe that by local parabolic regularity away from $x=0$, these spaces 
coincide with classical H\"older spaces in the interior, in addition to capturing the boundary behaviour 
which is needed for the analysis of the Yamabe flow.

We note that by using the product rule it is easy to see that $\hhok$ is closed under multiplication.  
We will also weight the spaces $\hhok$ with powers of $t$, note that $u \in t^\delta \hhok$ if and 
only if $u = t^\delta \widetilde{u}$ for an element $\widetilde{u} \in \hhok$.  

\begin{thm}\label{boundedness}
Let $(M^m,g)$ be a feasible edge space  with edge data $(B^b,F^f,\phi)$ 
and a singular neighbourhood $\calU\subset M$ of the edge $B$.
The incomplete edge metric is given by $g=g_0+h$ with 
\[
\left. g_0 \right|_{\calU} = dx^2 + \phi^*\ha +  x^2 \ka.
\]
Let $\lambda_0>f$ be the smallest non-zero eigenvalue of $\Delta_{\kappa, y},y\in B$. 
Put 
\[
 \A_0 := -\frac{(f-1)}{2}+\sqrt{\frac{(f-1)^2}{4}+\lambda_0}-1>0.
\]
Denote by $\Delta_g$ the Friedrichs extension of the Laplacian on $(M,g)$. Then for $\A\in (0,\A')$,
where $\A':=_0\min\{\A_0,1\}$, the heat operator $e^{-t\Delta_g}$ acts as a bounded convolution operator 
\begin{align*}
&e^{-t\Delta_g}: \ho \to \hho, \\
&e^{-t\Delta_g}: \ho \to \sqrt{t}\ho.
\end{align*}

\end{thm}

\begin{proof}
The proof we present is modeled on the classical estimation given in \cite{LSU:LAQ}.  Given $f \in \ho$, we use the asymptotics of the heat kernel to estimate each part of the H\"older norm of $e^{-t\Delta_g} * f$.  We break this into a number of steps, performing the more complicated estimations first, since they introduce techniques that are used throughout the proof.  By freezing the time variable we first estimate the H\"older seminorms in space, then by freezing the space variables we estimate the H\"older seminorms in time.  Note that in each of these steps we will have to localize and consider arguments near each corner of the blow-up heat space.  We finish by estimating the supremum norms at the end of the proof.

We will make frequent use of the fact that the edge heat kernel is stochastically complete, i.e. satisfies
\begin{align}\label{stoch-compl}
 \int_M H( t, p, \widetilde{p} ) \dv(\widetilde{p}) = 1, \; \mbox{for all} \; p \in M, t > 0, 
\end{align}
which allows us to introduce H\"older differences when needed.  This is a simple consequence of uniqueness of solutions to the heat equation. More precisely, with the Friedrichs extension $\Delta_g$ being a self-adjoint unbounded operator 
in the Hilbert space $L^2(M, \textup{vol}(g))$, the solutions $u(t,\cdot) \in \dom(\Delta_g)$ 
to the initial value problem
\begin{align*}
\partial_t u+\Delta_g u=0, \ u(0)=u_0\in \dom (\Delta_g) \subset L^2(M, \textup{vol}(g)),
\end{align*}
are unique and in fact given by $u(t)=e^{-t\Delta_g}u_0\in \dom (\Delta_g)$ for any $t>0$. 
Consequently, $u\equiv 1$ is the unique solution to the heat equation with the initial value 
$1\in \dom (\Delta_g)$, given in terms of the heat operator by $u\equiv 1=e^{-t\Delta_g}1$, 
note $\dim F\geq 1$. This is precisely the statement of stochastic completeness \eqref{stoch-compl}. 
The fact that $1\in \dom (\Delta_g)$ follows from \cite[\S 2.5]{MazVer:ATO} and is also elaborated in 
\cite{BDV:THE}.

Finally, we remark that the main idea of the proof in each coordinate system is to compare the asymptotics of the differentiated heat kernel with the powers of defining functions that arise from pulling the volume form back in coordinates.  The most difficult estimations that arise will leave an apparently singular factor of a defining function which is then appropriately estimated. Our proof has the following structure
\begin{center}
\begin{enumerate}
\item[] \S \ref{ho-space} Estimation of the H\"older difference in space.
\item[] \S \ref{ho-time} Estimation of the H\"older difference in time.
\item[] \S \ref{ho-sup} Estimation of the Supremum Norm.
\end{enumerate}
\end{center}

\subsection{Estimation of the H\"older difference in space}\label{ho-space} \ \\[-1mm]

Consider now for any $X\in \{\Delta_g, \ x^{-1}\mathcal{V}'_e\}$, $f\in \ho$ and any $p,p'\in M$
\begin{align*}
&X e^{-t\Delta_g}*f(t,p)-X e^{-t\Delta_g}*f(t,p')\\
=&\int_0^t\int_{M^+} X H(t-\wt, p, \widetilde{p})\left[ f(\wt, \widetilde{p})- f(\wt, p)\right] \,  \dv (\widetilde{p}) \, d\wt\\
-&\int_0^t\int_{M^+} X H(t-\wt, p', \widetilde{p})\left[ f(\wt, \widetilde{p})- f(\wt, p')\right] \,  \dv (\widetilde{p}) \, d\wt\\
+&\int_0^t\int_{M^-} \left[ X H(t-\wt, p, \widetilde{p}) -X H(t-\wt, p', \widetilde{p})\right] 
\left[ f(\wt, \widetilde{p})- f(\wt, p)\right] \,  \dv (\widetilde{p}) \, d\wt\\
-&\int_0^t\int_{M^-} X H(t-\wt, p', \widetilde{p})\left[ f(\wt, p)- f(\wt, p')\right] \,  \dv (\widetilde{p}) \, d\wt\\
=:&I_1-I_2+I_3-I_4,
\end{align*}
where we used \eqref{stoch-compl} and introduced the following notation
\begin{align*}
M^+:=\{\widetilde{p}\in M \mid d_M(p,\widetilde{p})\leq 2 d_M(p,p')\}, \\
M^-:=\{\widetilde{p}\in M \mid d_M(p,\widetilde{p})\geq 2 d_M(p,p')\}.
\end{align*}

Below we will lift the densities to the blowup space $\mathscr{M}^2_h$. 
We will identify the integration regions $M$, $M^+$ and $M^-$ with their corresponding lifts
without further comment.
We denote the four integrals above by $I_1, I_2, I_3$ and $I_4$ in the order of their appearance and 
estimate these integrals separately. Note that in case of $X=\partial_t$, we can employ the heat 
equation 
\[
\partial_t e^{-t\Delta_g} * f = - \Delta_g (e^{-t\Delta_g} * f) + f. 
\]
Hence indeed without loss of generality we can consider $X\in \{\Delta_g, \ x^{-1}\mathcal{V}'_e\}$. 
Moreover, we may assume $\A_0=\A'\leq 1$ without loss of generality, since for $\A_0>1$ we might as well set 
$\A_0=1$ in the estimates below.
\ \\[2mm]
\underline{\emph{Estimation of the first and second integrals $I_1, I_2$.}}\bigskip

We lift the heat kernel to a polyhomogeneous distribution $\beta^*H$ on the 
parabolic blowup of the heat space $\mathscr{M}^2_h$. 
The estimates in the interior of $\mathscr{M}^2_h$ are classical and hence we may assume 
that $\beta^*H$ is compactly supported in an open neighbourhood of the front face and 
estimate $I_1$ and $I_2$ near the various corners of ff.  
For any $X\in \{\Delta_g, \ x^{-1}\mathcal{V}_e\}$, we find 
as a crude consequence of Corollary \ref{corr}
\begin{align}
\label{XH}
\beta^*(XH)=(\rho_{\ff}\rho_{\td})^{-m-2} \rho_{\tf}^{\infty} G
\end{align}
where $G$ is a bounded polyhomogeneous distribution on $\mathscr{M}^2_h$.
We write down the estimates for $I_1$. The second integral $I_2$ is estimated along the same lines.
\\[3mm]
\textbf{Estimates near the lower left corner of the front face:}
Let us assume that the heat kernel $H$ is compactly supported near the lower left corner of the front face. 
Its asymptotic behaviour is appropriately described in the following projective coordinates
\begin{align*}
\tau=\frac{t-\wt}{x^2}, \ s=\frac{\wx}{x}, \ u=\frac{y-\widetilde{y}}{x}, \ x, \ y, \ z, \ \widetilde{z},
\end{align*}
where in these coordinates $\tau, s, x$ are the defining functions of tf, lf and ff respectively. 
The coordinates are valid whenever $(\tau, s)$ are bounded as $(t-\wt,x,\wx)$ approach zero. 
For the transformation rule of the volume form we compute
\begin{align*}
\beta^*(d\wt \dv(\wx, \wy, \wz)) =  h \cdot x^{m+2} s^f d\tau \, ds \, du \, d\wz,
\end{align*}
where $h$ is a bounded distribution on $\mathscr{M}^2_h$. 
Hence, using \eqref{XH} we arrive for any $f\in \ho$ after 
 cancellations at the estimate
\begin{equation*}
\begin{split}
|I_1|\leq C \, \|f\|_{\A} \, d_M((x,y,z), (x',y',z'))^{\A}.
\end{split}
\end{equation*}
\textbf{Estimates near the lower right corner of the front face:}
Let us assume that the heat kernel $H$ is compactly supported near the lower right corner of the front face. 
Its asymptotic behaviour is appropriately described in the following projective coordinates
\begin{align*}
\tau=\frac{t-\wt}{\wx^2}, \ s=\frac{x}{\wx}, \ u=\frac{y-\widetilde{y}}{\wx}, \ z, \wx, \ y, \ \widetilde{z},
\end{align*}
where in these coordinates $\tau, s, \wx$ are the defining functions of tf, rf and ff respectively. 
The coordinates are valid whenever $(\tau, s)$ are bounded as $(t-\wt,x,\wx)$ approach zero. 
For the transformation rule of the volume form we compute
\begin{align*}
\beta^*(d\wt \dv(\wx, \wy, \wz))=h \cdot \wx^{m+1} d\tau \, d\wx \, du \, d\wz,
\end{align*}
where $h$ is a bounded distribution on $\mathcal{M}^2_h$. 
Hence, using \eqref{XH} we find for any $f\in \ho$ after  cancellations
\begin{equation*}
\begin{split}
|I_1|\leq  C \|f\|_{\A} \int \int_{M^+} \wx^{-1} d_M((x,y,z), (\wx, \wy, \wz))^{\A} G \, d\tau \, d\wx \, du \, d\wz
\end{split}
\end{equation*}

The integration region lies within $M^+\cap \{x \leq \wx\}$, 
since we assume that the heat kernel is compactly supported near 
the lower right corner of the front face with $s<1$. Put
$d:=  d_M((x,y,z), (x',y',z'))$. Then for any $\wx \in M^+\cap \{x\leq \wx\}$ we find 
$|x-\wx|\leq 2d$ and $x\leq \wx$. Consequently
\begin{equation*}
\begin{split}
|I_1|&\leq  c \,  \|f\|_{\A} \int \int_x^{x+2d} \wx^{-1} d_M((x,y,z), (\wx, \wy, \wz))^{\A} G \, d\tau \, d\wx \, du \, d\wz \\
&\leq c' \, \|f\|_{\A} \int \int_x^{x+2d} \wx^{-1+\A} \left(|s-1|^2+|s+1|^2|z-\wz| + |u|^2\right)^{\frac{\A}{2}} G \, d\tau \, d\wx \, du \, d\wz \\
&\leq c'' \,  \|f\|_{\A} \left[(2d_M((x,y,z), (x',y',z'))+x)^{\A}-x^{\A}\right] \\
&= c'' \, \|f\|_{\A} \, d_M((x,y,z), (x',y',z'))^{\frac{\A}{2}} \left(
\left(2+x/d\right)^{\A}-\left(x/d\right)^{\A}\right).
\end{split}
\end{equation*}
The function $w(r)=(2+r)^\A-r^\A$ is bounded. 
Indeed 
\begin{align*}
(2+r)^\A-r^\A = r^\A \left(\left(1+\frac{2}{r}\right)^\A - 1\right) = 
2 r^{\A-1} \sum_{k=0}^\infty \binom{\A}{k} \left(\frac{2}{r}\right)^k \frac{\A-k}{k+1}= O(r^{\A-1}), 
\ r\to \infty.
\end{align*}  
Since $\A\in (0,1)$, boundedness follows and in we deduce
\begin{equation*}
\begin{split}
|I_1|\leq C \, \|f\|_{\A} \, d_M((x,y,z), (x',y',z'))^{\A}.
\end{split}
\end{equation*}
\\[3mm] \textbf{Estimates near the top corner of the front face:} 
Let us assume that the heat kernel $H$ is compactly supported near the top corner of the front face. 
Its asymptotic behaviour is appropriately described in the following projective coordinates
\begin{align*}
\rho=\sqrt{t-\wt}, \  \xi=\frac{x}{\rho}, \ \widetilde{\xi}=\frac{\widetilde{x}}{\rho}, 
\ u=\frac{y-\widetilde{y}}{\rho}, \ y, \ z, \ \widetilde{z},
\end{align*}
where in these coordinates $\rho, \xi, \widetilde{\xi}$ are the defining functions of the faces ff, lf and rf respectively. 
The coordinates are valid whenever $(\rho, \xi,\widetilde{\xi})$ are bounded as $(t-\wt,x,\w)$ approach zero. 
For the transformation rule of the volume form we compute
\begin{align*}
\beta^*(d\wt \dv(\wx, \wy, \wz))=h \cdot \rho^{m+1} \, \widetilde{\xi}^f \,  d\rho \, d\widetilde{\xi}\, du\, d\wz.
\end{align*}
where $h$ is a bounded distribution on $\mathcal{M}^2_h$. 
Hence, using \eqref{XH} and the fact that $G$ vanishes to infinite order as $|u| \to \infty$, 
we find for any $f\in \ho$ after  cancellations
\begin{equation*}
\begin{split}
|I_1| &\leq  c \, \|f\|_{\A} \int \int_{M^+} \rho^{-1+\A} \sqrt{(\xi-\widetilde{\xi})^2 + (\xi+\widetilde{\xi})^2 |z-\wz|^2 + |u|^2}^{\, \A} \, 
\widetilde{\xi}^f \,  G \, d\rho \, d\widetilde{\xi}\, du\, d\wz \\
&\leq c' \, \|f\|_{\A} \int \int_{M^+} \rho^{-1+\A} \, \widetilde{\xi}^f \, d\rho \, d\widetilde{\xi} 
= c' \, \|f\|_{\A}  \int \int_{M^+}  \wx^f \, \rho^{-2-f+\A} d\rho \, d\wx. 
\end{split}
\end{equation*}
We integrate within $\{\rho \geq \wx\}$, since $\widetilde{\xi}\leq 1$ near the top corner. 
Moreover, over $M^+$ we clearly have $|x-\wx|\leq 2d_M((x,y,z), (x',y',z'))=:2d$. Hence we find 
\begin{align*}
|I_1| &\leq c'' \, \|f\|_{\A}  \int_{|x-\wx|\leq 2d}  \wx^f \, \int_{\wx}^{\infty} \rho^{-2-f+\A} d\rho \, d\wx
\leq c'' \, \|f\|_{\A}  \int_{|x-\wx|\leq 2d}  \wx^{-1+\A} d\wx 
\end{align*}
We consider the cases $x\geq \wx$ and $x\leq \wx$ separately and find
\begin{equation*}
|I_1| \leq \left\{
\begin{split}
& c''' \, \|f\|_{\A} d^{\A} ((2+x/d)^\A-(x/d)^\A), \ \textup{if} \ x\leq \wx, \\
& c''' \, \|f\|_{\A} d^{\A} ((x/d)^\A-(x/d-2)^\A), \ \textup{if} \ x\geq \wx. 
\end{split} \right.
\end{equation*}
The functions $w(r)=(2+r)^\A-r^\A$ and $w(r)=r^\A - (r-2)^\A$ are both bounded as before,
and hence in both cases we deduce
\begin{equation*}
\begin{split}
|I_1|\leq C \, \|f\|_{\A} \, d_M((x,y,z), (x',y',z'))^{\A}.
\end{split}
\end{equation*}
\textbf{Estimates where the diagonal meets the front face:} 
We proceed using projective coordinates $(\tau, s, u)$ near the left corner, which are 
valid near td as well. $\eta:=\sqrt{\tau}$ is the defining function of td and $x$ the 
defining function of ff. For the transformation rule of the volume form we compute
\begin{align*}
\beta^*(d\wt \dv(\wx, \wy, \wz))=h \cdot \eta \, x^{m+2} s^f d\eta \, ds \, du \, d\wz,
\end{align*}
where $h$ is a bounded distribution on $\mathcal{M}^2_h$. 
Hence, using \eqref{XH} we find for any $f\in \ho$ after  cancellations
\begin{equation*}
\begin{split}
|I_1|\leq  c \|f\|_{\A} \int_{M^+} \eta^{-m-1} d_M((x,y,z), (\wx, \wy, \wz))^{\A} \, G \, d\eta \, ds \, du \, d\wz.
\end{split}
\end{equation*}

We estimate the distance 
\begin{equation*}
\begin{split}
d_M((x,y,z), (\wx, \wy, \wz)) &= \sqrt{|x-\wx|^2+(x+\wx)^2|z-\wz|^2 + |y-\wy|^2} \\
&= x \, \sqrt{|1-s|^2+(1+s)^2|z-\wz|^2 + |u|^2} \\ 
&\geq  x \, \sqrt{|1-s|^2+|z-\wz|^2 + |u|^2} =: x \, r(s,z,\wz, u).
\end{split}
\end{equation*}

On the other hand
\begin{equation*}
\begin{split}
d_M((x,y,z), (\wx, \wy, \wz))
&= x \, \sqrt{|1-s|^2+(1+s)^2|z-\wz|^2 + |u|^2} \\ 
&\leq 2 \, x \, \sqrt{|1-s|^2+|z-\wz|^2 + |u|^2} =2\, x \, r(s,z,\wz, u).
\end{split}
\end{equation*}

We transform $(s,u,\wz)$ around $(1,0,z)$ to polar coordinates with $r$ as the radial coordinate. 
After integration in the angular coordinates and substituting $\sigma = \eta / r$, we have
\begin{equation}
\label{sigma}
\begin{split}
|I_1|&\leq  c'' \|f\|_{\A} x^\A \int \int_{M^+} \sigma^{-m-1} r^{-1+\A} \, G \, d\sigma \, dr \\
&\leq c'' \|f\|_{\A} \, x^\A \int r^{-1+\A} \, dr  \int_0^{\infty}\sigma^{-m-1} \, G \, d\sigma,
\end{split}
\end{equation}
where the integration region in $r$ lies within $\{r\leq 2d_M((x,y,z), (x', y', z')) / x\}$.
The variable $\sigma$ is given in terms of the projective coordinates $(S,U,Z)$ near td, see \eqref{d-coord} 
as follows
\[
\sigma^{-1} = \frac{r}{\eta}=\sqrt{\left(\frac{x-\wx}{\sqrt{t-\wt}}\right)^2 +  
\left(\frac{y-\wy}{\sqrt{t-\wt}}\right)^2 + 
\left(\frac{\wx(z-\wz)}{\sqrt{t-\wt}}\right)^2} = \sqrt{|S|^2+|U|^2+|Z|^2}.
\]

Consequently the $\sigma$ integral in \eqref{sigma} is bounded, since $G$ is vanishing 
to infinite order as $|(S,U,Z)|\to \infty$. As the integration region in $r$ is within 
$\{r\leq 2d/x\}$, where $d:=d_M((x,y,z), (x', y', z'))$
\begin{equation*}
\begin{split}
|I_1| \leq  c''' \|f\|_{\A} \,  x^\A \int_0^{2d/x} r^{-1+\A} \, dr \leq C \, \|f\|_{\A} \, d_M((x,y,z), (x', y', z'))^\A.
\end{split}
\end{equation*}
\ \\[2mm]

\underline{\emph{Estimation of the third integral $I_3$.}}\bigskip

As in the previous estimates we lift the heat kernel to a polyhomogeneous distribution $\beta^*H$ on the 
parabolic blowup of the heat space $\mathscr{M}^2_h$, and
assume that $\beta^*H$ is compactly supported in an open neighbourhood of the front face.
Using the standard coordinates and the mean value theorem we obtain

\begin{equation*}
\begin{split}
I_3 &=  |x-x'|\int_0^t\int_{M^-} \partial_{\xi} X H(t-\wt, \xi, y, z, \wx, \wy, \wz) 
\left[ f(\wt, \wx, \wy, \wz)- f(\wt, x, y, z)\right] \,  \dv \, d\wt \\
&+\,  |y-y'|\int_0^t\int_{M^-} \partial_{\gamma} X H(t-\wt, x', \gamma, z, \wx, \wy, \wz) 
\left[ f(\wt, \wx, \wy, \wz)- f(\wt, x, y, z)\right] \,  \dv \, d\wt \\
&+\, |z-z'|\int_0^t\int_{M^-} \partial_{\zeta} X H(t-\wt, x', y', \zeta, \wx, \wy, \wz) 
\left[ f(\wt, \wx, \wy, \wz)- f(\wt, x, y, z)\right] \,  \dv \, d\wt,
\end{split}
\end{equation*}
for some intermediate values $(\xi, \gamma, \zeta)$.
Consider any $(x'',y'',z'')\in M$ with 
$$
d_M(x,y,z,x'',y'',z'')\leq d_M(x,y,z,x',y',z').
$$
Then for any $(\wx,\wy,\wz)\in M^-$ we compute 

\begin{equation*}
\begin{split}
2d_M(x,y,z,x',y',z') &\leq d_M(x, y, z,\wx, \wy, \wz) \\
&\leq d_M(x,y,z,x'',y'',z'') + d_M(x'',y'',z'',\wx, \wy, \wz) \\
&\leq d_M(x,y,z,x',y',z') + d_M(x'',y'',z'',\wx, \wy, \wz).
\end{split}
\end{equation*}
We deduce 

\begin{equation}
\label{inequa}
\begin{split}
d_M(x,y,z,x',y',z') &\leq d_M(x'',y'',z'',\wx, \wy, \wz),\\
d_M(x, y, z,\wx, \wy, \wz) &\leq 2 d_M(x'',y'',z'',\wx, \wy, \wz).
\end{split}
\end{equation}
Using the last inequality we may estimate $I_3$ for any $f\in \ho$ as follows
 
\begin{equation*}
\begin{split}
|I_3| &\leq 2 \, \|f\|_{\A} |x-x'| \int_0^t\int_{M^-} |\partial_{\xi} X H(t-\wt, \xi, y, z, \wx, \wy, \wz) |
d_M(\xi, y, z, \wx, \wy, \wz)^\A \,  \dv \, d\wt \\
&+\, 2 \, \|f\|_{\A} |y-y'|\int_0^t\int_{M^-} |\partial_{\gamma} X H(t-\wt, x', \gamma, z, \wx, \wy, \wz) |
d_M(x', \gamma, z, \wx, \wy, \wz)^\A \,  \dv \, d\wt \\
&+\, 2 \, \|f\|_{\A}|z-z'|\int_0^t\int_{M^-} |\partial_{\zeta} X H(t-\wt, x', y', \zeta, \wx, \wy, \wz) |
d_M(x', y', \zeta, \wx, \wy, \wz)^\A \,  \dv \, d\wt \\
&=: I'_3 + I''_3+I'''_3.
\end{split}
\end{equation*}

\ \\[1mm]
\textbf{Estimates near the lower left corner of the front face:}
Let us assume that the integrand in each of the components of $I_3$ is compactly supported 
near the lower left corner of the front face. Their asymptotic behaviour is appropriately 
described in the following projective coordinates
\begin{equation}
\label{lf-coordinates-I}
\begin{split}
&(I'_3) \quad \tau=\frac{t-\wt}{\xi^2}, \ s=\frac{\wx}{\xi}, \ u=\frac{y-\widetilde{y}}{\xi}, \ \theta=\xi, \ y, \ z, \ \widetilde{z}, \\
&(I''_3) \quad \tau=\frac{t-\wt}{(x')^2}, \ s=\frac{\wx}{x'}, \ u=\frac{\gamma-\widetilde{y}}{x'}, \ \theta= x', \ y, \ z, \ \widetilde{z}, \\
&(I'''_3) \quad \tau=\frac{t-\wt}{(x')^2}, \ s=\frac{\wx}{x'}, \ u=\frac{y'-\widetilde{y}}{x'}, \ \theta=x', \ y, \ \zeta, \ \widetilde{z}, \\
\end{split}
\end{equation}
where in these coordinates $\tau, s, \theta$ are the defining functions of tf, rf and ff respectively. 
The coordinates are valid whenever $(\tau, s)$ are bounded as $(t-\wt,\theta,\wx)$ approach zero. 
For the transformation rule of the volume form we compute
\begin{align*}
\beta^*(d\wt \dv(\wx, \wy, \wz)) =  h \cdot \theta^{m+2} d\tau \, ds \, du \, d\wz,
\end{align*}
where $h$ is a bounded distribution on $\mathscr{M}^2_h$. 
Hence, using \eqref{XH}, and regarding $(\wx,\wy)$ as functions of $(s,u, \xi)$, we arrive for any $f\in \ho$ after 
 cancellations at the estimate (note, differentiation in $\zeta$ does not lower the front face asymptotics)
\begin{equation*}
\begin{split}
|I'_3| &\leq c \, \|f\|_{\A} \, |x-x'| \, \xi^{-1+\A} \int \int_{M^-} (d_M(\xi, y, z,\wx, \wy, \wz)/\xi)^{\A} \, G \, d\tau \, ds \, du \, d\wz, \\
|I''_3| &\leq c \, \|f\|_{\A} \, |y-y'| \, (x')^{-1+\A} \int \int_{M^-} (d_M(x', \gamma, z,\wx, \wy, \wz)/x')^{\A} \, G \, d\tau \, ds \, du \, d\wz, \\
|I'''_3| &\leq c \, \|f\|_{\A} \, |z-z'| \, (x')^\A \int \int_{M^-} (d_M(x', y', \zeta,\wx, \wy, \wz)/x')^{\A} \, G \, d\tau \, ds \, du \, d\wz.
\end{split}
\end{equation*}
Note that in each case we can estimate the integrand using the appropriate projective coordinates
\begin{equation*}
\left.
\begin{split}
(I'_3) \quad &d_M(\xi, y, z,\wx, \wy, \wz)/\xi \leq 2 \sqrt{|1-s|^2+|z-\wz|+|u|^2} \\
(I''_3) \quad &d_M(x', \gamma, z,\wx, \wy, \wz)/x' \leq 2 \sqrt{|1-s|^2+|z-\wz|+|u|^2} \\
(I'''_3) \quad &d_M(x', y', \zeta, \wx, \wy, \wz)/x' \leq 2 \sqrt{|1-s|^2+|\zeta-\wz|+|u|^2}
\end{split}
\right\}
=:2 r(s,u,\wz).
\end{equation*}

By \eqref{inequa} the integration region $M^-$ lies within $\{r\geq d_M(x,y,z,x',y',z')/\theta\}$.
If $u$ is bounded, then $r$ is bounded near the left corner, so that $1 \leq c r^{-1}$ and hence we may estimate constants against as many negative powers of $r$ 
as we like. Similarly, if $|u|$ tends to infinity, then $r/|u|$ is bounded near the left corner, and we may introduce as 
many negative powers of $r/|u|$ as we like. The additional powers of $|u|$ are irrelevant, since $G$ vanishes to infinite order 
as $|u|\to \infty$. Hence we may estimate 
\begin{equation*}
\begin{split}
|I'_3| &\leq c' \, \|f\|_{\A} \, |x-x'| \, \xi^{-1+\A} (d_M(x,y,z,x',y',z')/\xi)^{-1+\A}, \\
|I''_3| &\leq c' \, \|f\|_{\A} \, |y-y'| \, (x')^{-1+\A} (d_M(x,y,z,x',y',z')/x')^{-1+\A}, \\
|I'''_3| &\leq c' \, \|f\|_{\A} \, |z-z'| \, (x')^\A (d_M(x,y,z,x',y',z')/x')^{-1+\A}.
\end{split}
\end{equation*}
Altogether we finally obtain
\begin{equation*}
\begin{split}
|I_3| &\leq c'' \, \|f\|_{\A} \, (|x-x'| +|y-y'| + x' |z-z'|)\, d_M(x,y,z,x',y',z')^{-1+\A} \\
& \leq c''' \, \|f\|_{\A} \, d_M(x,y,z,x',y',z')^{\A}.
\end{split}
\end{equation*}
\textbf{Estimates near the lower right corner of the front face:}
Let us assume that the integrand in each of the components of $I_3$ is compactly supported 
near the lower right corner of the front face. Their asymptotic behaviour is appropriately 
described in the following projective coordinates
\begin{equation*}
\begin{split}
&(I'_3) \quad \tau=\frac{t-\wt}{\wx^2}, \ s=\frac{\xi}{\wx}, \ u=\frac{y-\widetilde{y}}{\wx}, \ \theta=\xi, \ \wx, \ y, \ z, \ \widetilde{z}, \\
&(I''_3) \quad \tau=\frac{t-\wt}{\wx^2}, \ s=\frac{x'}{\wx}, \ u=\frac{\gamma-\widetilde{y}}{\wx}, \ \theta=x', \ \wx, \ y, \ z, \ \widetilde{z}, \\
&(I'''_3) \quad \tau=\frac{t-\wt}{\wx^2}, \ s=\frac{x'}{\wx}, \ u=\frac{y'-\widetilde{y}}{\wx}, \ \theta=x', \ \wx, \ y, \ \zeta, \ \widetilde{z}, \\
\end{split}
\end{equation*}
where in these coordinates $\tau, s, \wx$ are the defining functions of tf, lf and ff respectively. 
The coordinates are valid whenever $(\tau, s)$ are bounded as $(t-\wt,\theta,\wx)$ approach zero. 
For the transformation rule of the volume form we compute
\begin{align*}
\beta^*(d\wt \dv(\wx, \wy, \wz)) =  h \cdot \wx^{m+1} d\tau \, d\wx \, du \, d\wz,
\end{align*}
where $h$ is a bounded distribution on $\mathscr{M}^2_h$. 
Hence, using the careful estimates in Corollary \ref{corr} we arrive for any $f\in \ho$ after cancellations at the estimate 
(note, differentiation in $\zeta$ does not lower the front face asymptotics but kills the fibrewise constant leading terms)
\begin{equation*}
\begin{split}
|I'_3| &\leq c \, \|f\|_{\A} \, |x-x'| \, \int \int_{M^-} (\wx^{-2+\A} s^{\A_0}G_1 + \wx^{-1+ \A} s^{-1+\A_0}G_2) 
 (d_M(\xi, y, z,\wx, \wy, \wz)/\wx)^{\A}, \\
|I''_3| &\leq c \, \|f\|_{\A} \, |y-y'| \, \int \int_{M^-} \wx^{-2+\A} (d_M(x', \gamma, z,\wx, \wy, \wz)/\wx)^{\A} \, G \, d\tau \, d\wx \, du \, d\wz, \\
|I'''_3| &\leq c \, \|f\|_{\A} \, |z-z'| \, \int \int_{M^-} \wx^{-1+\A} s^{\A_0} (d_M(x', y', \zeta,\wx, \wy, \wz)/\wx)^{\A} \, G \, d\tau \, d\wx \, du \, d\wz,
\end{split}
\end{equation*}
with bounded polyhomogeneous kernels $G_1,G_2, G$.
We begin with the estimation of $I''_3$. Put $d:=d_M(x,y,z,x',y',z')$.
We separate the integration region into $M^-_1=M^-\cap \{\wx \leq d\}$ and $M^-_2=M^-\cap \{\wx \geq d\}$. Then 
using $|y-y'|\leq d$ and \eqref{inequa} we find
\begin{equation*}
\begin{split}
|I''_3| &\leq c' \, \|f\|_{\A} \, |y-y'| \, \int \int_{M^-_1} \wx^{-2+\A} (d_M(x', \gamma, z,\wx, \wy, \wz)/\wx)^{\A} \, G \, d\tau \, d\wx \, du \, d\wz \\
        & + c' \, \|f\|_{\A} \, |y-y'| \, \int \int_{M^-_2} \wx^{-2+\A} (d_M(x', \gamma, z,\wx, \wy, \wz)/\wx)^{\A} \, G \, d\tau \, d\wx \, du \, d\wz \\
        &\leq c'' \, \|f\|_{\A}\, \int \int_{M^-_1} \wx^{-1+\A} (d_M(x', \gamma, z,\wx, \wy, \wz)/\wx)^{1+\A} \, G \, d\tau \, d\wx \, du \, d\wz \\
        & + c'' \, \|f\|_{\A} \, d \, \int \int_{M^-_2} \wx^{-2+\A} (d_M(x', \gamma, z,\wx, \wy, \wz)/\wx)^{\A} \, G \, d\tau \, d\wx \, du \, d\wz.
\end{split}
\end{equation*}
Note that in the corresponding projective coordinates
\begin{equation}
\label{inequa2}
d_M(x', \gamma, z,\wx, \wy, \wz)/\wx \leq 2 \sqrt{|1-s|^2+|z-\wz|+|u|^2}=:2 r(s,u,\wz).
\end{equation}

If $u$ is bounded, then $r$ is bounded near the right corner. 
If $|u|$ tends to infinity, then $r/|u|$ is bounded near the right corner, 
and the additional powers of $|u|$ get absorbed by $G$ which vanishes 
to infinite order as $|u|\to \infty$. Hence we can estimate $\int rG \, du$ 
against a constant and find
\begin{equation*}
\begin{split}
|I''_3| \leq c''' \, \|f\|_{\A}\, \left( \int_0^d \wx^{-1+\A} \, d\wx + d \, \int_d^{\infty} \wx^{-2+\A} \, d\wx\right)
\leq C \, \|f\|_{\A} \, d^\A.
\end{split}
\end{equation*}
In order to estimate $I'''_3$, note that $s \wx = x'$ and hence 
\begin{equation*}
\begin{split}
|I'''_3| &\leq c' \, \|f\|_{\A} \, (x')^{\A_0}|z-z'| \, \left|\int_{x'}^{\infty} \wx^{-1-\A_0 +\A} d\wx \right|
\leq c'' \, \|f\|_{\A} \, |z-z'| \, (x')^{\A} \\
&\leq c'' \, \|f\|_{\A} \, (x+x')^{\A_0}|z-z'|^\A \leq C \, \|f\|_{\A} \, d^\A.
\end{split}
\end{equation*}
It remains to estimate $I'_3$.
\begin{equation*}
\begin{split}
|I'_3| & \leq c \, \|f\|_{\A} \, |x-x'| \, \int \int_{M^-} \wx^{-2+\A} s^{\A_0} (d_M(\xi, y, z,\wx, \wy, \wz)/\wx)^{\A} \, G_1 \, d\tau \, d\wx \, du \, d\wz \\
& + c \, \|f\|_{\A} \, |x-x'| \, \int \int_{M^-} \wx^{-1+\A} s^{-1+\A_0} (d_M(\xi, y, z,\wx, \wy, \wz)/\wx)^{\A} \, G_2 \, d\tau \, d\wx \, du \, d\wz.
\end{split}
\end{equation*}

The first integral is estimated precisely as $I''_3$. 
For the second integral we find (note $s\wx=\xi$)
using \eqref{inequa}
\begin{equation*}
\begin{split}
&c' \, \|f\|_{\A} \, |x-x'| \, \int \int_{M^-} \wx^{-1+\A} s^{-1+\A_0} (d_M(\xi, y, z,\wx, \wy, \wz)/\wx)^{\A} \, G_2 \, d\tau \, d\wx \, du \, d\wz \\
= \, &c'\|f\|_\A \, |x-x'|^{\A_0} \left( \frac{|x-x'|}{\xi} \right)^{1-\A_0} \int \int_{M^-} 
\wx^{-\A_0+\A} (d_M(\xi, y, z,\wx, \wy, \wz)/\wx)^{\A} \, G_2 \, d\tau \, d\wx \, du \, d\wz \\
\leq \, &c''\|f\|_\A \, |x-x'|^{\A_0} \int \int_{M^-} 
\wx^{-\A_0+\A} (d_M(\xi, y, z,\wx, \wy, \wz)/\wx)^{1-\A_0 +\A} \, G_2 \, d\tau \, d\wx \, du \, d\wz \\
\leq \, &c'''\|f\|_\A \, |x-x'|^{\A_0} \int_{0}^{1} \wx^{-\A_0+\A} d\wx 
\leq C \, \|f\|_{\A} \, d_M(x,y,z,x',y',z')^\A.
\end{split}
\end{equation*}
\textbf{Estimates near the top corner of the front face:}
Let us assume that the heat kernel $H$ is compactly supported near the top corner of the front face. 
Its asymptotic behaviour is appropriately described in the following projective coordinates
\begin{equation*}
\begin{split}
&(I'_3) \quad \rho=\sqrt{t-\wt}, \  \nu=\frac{\xi}{\rho}, \ \widetilde{\nu}=\frac{\wx}{\rho}, \ u=\frac{y-\wy}{\rho}, \ \theta=\xi, \ y, \ z, \ \wz, \\
&(I''_3) \quad \rho=\sqrt{t-\wt}, \  \nu=\frac{x'}{\rho}, \ \widetilde{\nu}=\frac{\wx}{\rho}, \ u=\frac{\gamma-\wy}{\rho}, \ \theta=x', \ \gamma, \ z, \ \wz, \\
&(I'''_3) \quad \rho=\sqrt{t-\wt}, \  \nu=\frac{x'}{\rho}, \ \widetilde{\nu}=\frac{\wx}{\rho}, \ u=\frac{y'-\wy}{\rho}, \ \theta=x',\ y', \ \zeta, \ \wz,
\end{split}
\end{equation*}
where in these coordinates $\rho, \nu, \widetilde{\nu}$ are the defining functions of the faces ff, lf and rf respectively. 
The coordinates are valid whenever $(\rho, \nu,\widetilde{\nu})$ are bounded as $(t-\wt,x, \xi, x', \wx)$ approach zero. 
For the transformation rule of the volume form we compute
\begin{align*}
\beta^*(d\wt \dv(\wx, \wy, \wz))=h \cdot \rho^{m+1} \, \widetilde{\nu}^f \,  d\rho \, d\widetilde{\nu}\, du\, d\wz.
\end{align*}
where $h$ is a bounded distribution on $\mathcal{M}^2_h$. 
As in the previous discussion near the right corner of the front face, we find for any $f\in \ho$ after  cancellations
\begin{equation*}
\begin{split}
|I'_3| &\leq c \, \|f\|_{\A} \, |x-x'| \, \int \int_{M^-} (\rho^{-2+\A} \nu^{\A_0}G_1 + \rho^{-1+\A} \nu^{-1+\A_0}G_1)  (d_M(\xi, y, z,\wx, \wy, \wz)/\wx)^{\A}, \\
|I''_3| &\leq c \, \|f\|_{\A} \, |y-y'| \, \int \int_{M^-} \rho^{-2+\A} (d_M(x', \gamma, z,\wx, \wy, \wz)/\wx)^{\A} \, G \, d\rho \, d\widetilde{\nu}\, du\, d\wz, \\
|I'''_3| &\leq c \, \|f\|_{\A} \, |z-z'| \, \int \int_{M^-} \rho^{-1+\A} \nu^{\A_0} (d_M(x', y', \zeta,\wx, \wy, \wz)/\wx)^{\A} \, G \, d\rho \, d\widetilde{\nu}\, du\, d\wz.
\end{split}
\end{equation*}
The integrals are estimated in exactly 
the same manner as near the right corner.

\ \\[2mm]\textbf{Estimates near the diagonal meets the front face:}
We proceed using projective coordinates \eqref{lf-coordinates-I} near the left corner, which are 
valid near td as well. $\eta:=\sqrt{\tau}$ is the defining function of td and $x$ the 
defining function of ff. For the transformation rule of the volume form we compute
\begin{align*}
\beta^*(d\wt \dv(\wx, \wy, \wz))=h \cdot \eta \, \theta^{m+2} s^f d\eta \, ds \, du \, d\wz,
\end{align*}
where $h$ is a bounded distribution on $\mathcal{M}^2_h$. 
Hence, using Corollary \ref{corr} we find for any $f\in \ho$ after  cancellations
(note, differentiation in $\zeta$ does not lower the front face asymptotics)
\begin{equation*}
\begin{split}
|I'_3| &\leq c \, \|f\|_{\A} \, |x-x'| \, \xi^{-1+\A} \int \int_{M^-} \eta^{-m-2} (d_M(\xi, y, z,\wx, \wy, \wz)/\xi)^{\A} \, G \, d\eta \, ds \, du \, d\wz, \\
|I''_3| &\leq c \, \|f\|_{\A} \, |y-y'| \, (x')^{-1+\A} \int \int_{M^-} \eta^{-m-2} (d_M(x', \gamma, z,\wx, \wy, \wz)/x')^{\A} \, G \, d\eta \, ds \, du \, d\wz, \\
|I'''_3| &\leq c \, \|f\|_{\A} \, |z-z'| \, (x')^\A \int \int_{M^-}  \eta^{-m-2} (d_M(x', y', \zeta,\wx, \wy, \wz)/x')^{\A} \, G \, d\eta \, ds \, du \, d\wz.
\end{split}
\end{equation*}
We estimate the distance term using appropriate projective coordinates
\begin{equation*}
\left.
\begin{split}
(I'_3) \quad &d_M(\xi, y, z,\wx, \wy, \wz)/\xi \leq 2 \sqrt{|1-s|^2 +|z-\wz|+|u|^2} \\
(I''_3) \quad &d_M(x', \gamma, z,\wx, \wy, \wz)/x' \leq 2 \sqrt{|1-s|^2 +|z-\wz|+|u|^2} \\
(I'''_3) \quad &d_M(x', y', \zeta, \wx, \wy, \wz)/x' \leq 2 \sqrt{|1-s|^2 +|\zeta-\wz|+|u|^2}
\end{split}
\right\}
=:2 r(s,u,\wz).
\end{equation*}

We transform $(s,u,\wz)$ around $(1,0,z)$ to polar coordinates with $r$ as the radial variable.  
After integration in the angular variables and substituting $\sigma = \eta / r$ we have
\begin{equation*}
\begin{split}
|I'_3| &\leq c \, \|f\|_{\A} \, |x-x'| \, \xi^{-1+\A} \int \int_{M^-} \sigma^{-m-2} r^{-2+\A} \, G \, d\sigma \, dr, \\
|I''_3| &\leq c \, \|f\|_{\A} \, |y-y'| \, (x')^{-1+\A} \int \int_{M^-} \sigma^{-m-2} r^{-2+\A} \, G \, d\sigma \, dr, \\
|I'''_3| &\leq c \, \|f\|_{\A} \, |z-z'| \, (x')^\A \int \int_{M^-}  \sigma^{-m-2} r^{-2+\A} \, G \, d\sigma \, dr.
\end{split}
\end{equation*}
where by \eqref{inequa} the integration region $M^-$ lies within $\{r\geq d_M(x,y,z,x',y',z')/\theta\}$. 
$\sigma$ is given in terms of the projective coordinates $(S,U,Z)$ near td, see \eqref{d-coord} 
as follows
\[
\sigma^{-1} = \sqrt{|S|^2+|U|^2+|Z|^2}.
\]

Consequently the $\sigma$ integral in \eqref{sigma} is bounded, since $G$ is vanishing 
to infinite order as $|(S,U,Z)|\to \infty$. We find, with the integration region in $r$ within 
$\{r\geq d_M(x,y,z,x',y',z')/\theta\}$
\begin{equation*}
\begin{split}
|I'_3| &\leq c' \, \|f\|_{\A} \, |x-x'| \, d_M(x,y,z,x',y',z')^{-1+\A} \\
|I''_3| &\leq c' \, \|f\|_{\A} \, |y-y'| \, d_M(x,y,z,x',y',z')^{-1+\A}, \\
|I'''_3| &\leq c' \, \|f\|_{\A} \, (x')|z-z'| \, d_M(x,y,z,x',y',z')^{-1+\A}.
\end{split}
\end{equation*}

Altogether we finally obtain
\begin{equation*}
\begin{split}
|I_3| &\leq c'' \, \|f\|_{\A} \, (|x-x'| +|y-y'| + x' |z-z'|)\, d_M(x,y,z,x',y',z')^{-1+\A} \\
& \leq c''' \, \|f\|_{\A} \, d_M(x,y,z,x',y',z')^{\A}.
\end{split}
\end{equation*}
\ \\[2mm]\underline{\emph{Estimation of the fourth integral $I_4$.}}\bigskip

First we employ stochastic completeness of the heat kernel and compute
\begin{equation*}
\begin{split}
I_4&=\int_0^t\int_{M^-} X H(t-\wt, p', \widetilde{p})\left[ f(\wt, p)- f(\wt, p')\right] \,  \dv (\widetilde{p}) \, d\wt \\
&= \int_0^t \left[ f(\wt, p)- f(\wt, p')\right]  \int_{M} X H(t-\wt, p', \widetilde{p}) \, \dv(\widetilde{p}) \,  d\wt \\
&-  \int_0^t \left[ f(\wt, p)- f(\wt, p')\right]  \int_{M^+} X H(t-\wt, p', \widetilde{p}) \, \dv(\widetilde{p})\,  d\wt \\
&= \int_0^t \left[ f(\wt, p)- f(\wt, p')\right] \int_{M^+} (- X H(t-\wt, p', \widetilde{p})) \, \dv(\widetilde{p})\,  d\wt. 
\end{split}
\end{equation*}

We lift the heat kernel to a polyhomogeneous distribution $\beta^*H$ on the 
parabolic blowup of the heat space $\mathscr{M}^2_h$, and as before 
assume that $\beta^*H$ is compactly supported in an open neighbourhood of the front face and 
estimate $I_4$ near the various corners of ff. Estimation near the lower left corner of the front face 
follows along the lines of $I_1$ estimates. The estimates near the top corner are parallel to 
those near the right corner. Therefore, we only explicate the estimates near the right corner and the diagonal. 

\ \\[2mm]\textbf{Estimates near the lower right corner of the front face:}
The asymptotic behaviour of $\beta^*H$ is appropriately described in the following projective coordinates
\begin{align*}
\tau=\frac{t-\wt}{\wx^2}, \ s=\frac{x'}{\wx}, \ u=\frac{y'-\widetilde{y}}{\wx}, \ z', \wx, \ y, \ \widetilde{z},
\end{align*}
where in these coordinates $\tau, s, \wx$ are the defining functions of tf, rf and ff respectively. 
The coordinates are valid whenever $(\tau, s)$ are bounded as $(t-\wt,x',\wx)$ approach zero.
For the transformation rule of the volume form we compute
\begin{align*}
\beta^*(d\wt \dv(\wx, \wy, \wz))=h \cdot \wx^{m+1} d\tau \, d\wx \, du \, d\wz,
\end{align*}
where $h$ is a bounded distribution on $\mathcal{M}^2_h$. 
By Corollary \ref{corr} we may write
$$
\beta^*(XH) = \wx^{-m-2}s^{\A_0} \tau^\infty G_1 + \wx^{-m-1} s^0 \tau^\infty G_2 + \beta^*(\partial_y  L H),
$$
with bounded polyhomogeneous kernels $G_1, G_2$.
We separate $I_4=I_4'+I_4''+ I_4'''$ correspondingly.
We find for any $f\in \ho$ after  cancellations
\begin{equation*}
\begin{split}
|I'_4| &\leq  c \, \|f\|_{\A} \, d_M((x,y,z), (x', y', z'))^{\A} \int_{M^+} \wx^{-1} s^{\A_0} G\, d\tau \, d\wx \, du \, d\wz \\
& \leq c' \, \|f\|_{\A} \, d_M((x,y,z), (x', y', z'))^{\A} (x')^{\A_0} \int_{x'}^{\infty} \wx^{-1-\A_0} d\wx \\
& \leq c'' \, \|f\|_{\A} \, d_M((x,y,z), (x', y', z'))^{\A}.
\end{split}
\end{equation*}

Estimates for $I_4''$ follow along the same lines. 
Consider $I_4'''$, where we exemplify the estimates by setting the integrand to $\partial_{y_j} LH$ for any $j=1,...,b$
and perform integration by parts in the corresponding $u_j\in \R$. For fixed $(\wx,\wz)$ the integration region is 
$u_j \in [-R(\wx,\wz),R(\wx,\wz)]$, where $|u_j|=R(\wx,\wz)$ corresponds to $(\wx,\wy,\wz)\in \partial M^+$. 
Write $\widehat{u}^{j}:=(u_1,..,\widehat{u}_j,..,u_b)\in \R^{b-1}$. We simplify the expressions by setting
$$
\delta f := f(t-\tau\wx^2, x,y,z)-f(t-\tau\wx^2, x',y',z').
$$
Integration by parts then gives
\begin{equation*}
\begin{split}
I'''_4 & = \int \delta f   \int_{M^+}\wx^{-1} \partial_{u_j} (\beta^*LH)\, \wx^{m+1} h(\wx, y-\wx u, \wz) \, d\tau \, d\wx \, du \, d\wz \\
& =   \int  \delta f \left. \int \wx^{m} \beta^*(LH) \, h(\wx, y-\wx u, \wz)
\right|_{|u_j|=R(\wx,\wz)} \, d\tau \, d\wx \, d\widehat{u}^j \, d\wz \\
& -  \int \delta f \int_{M^+}\wx^{m} \beta^*(LH) \, \partial_{u_j} h(\wx, y-\wx u, \wz) \, d\tau \, d\wx \, du \, d\wz \\
& = A + B.
\end{split}
\end{equation*}
Since $h$ is smooth in the edge variable, $\partial_u h(\wx, y-\wx u, \wz)=\wx \partial_y h(\wx, y-\wx u, \wz)$.
Hence, the integrand in $B$ is bounded in $\wx$ is and we deduce
after cancellations
$$
|B| \leq c \, \|f\|_{\A} \, d_M((x,y,z), (x', y', z'))^{\A}.
$$
For the first integral we note that for $|u_j|=R(\wx,\wz)$ we have $(\wx,\wy,\wz)\in \partial M^+$ so that 
\begin{align*}
&d_M((x,y,z), (\wx,\wy,\wz)) = 2 d_M((x,y,z), (x', y', z')), \\
&d_M((x', y', z'),(\wx,\wy,\wz)) \geq d_M((x,y,z), (x', y', z')),
\end{align*}
where the second relation follows from the first one using the triangle inequality. Using the second relation we find
\begin{equation*}
\begin{split}
|A| & \leq c \, \|f\|_{\A} \, d_M((x,y,z), (x', y', z'))^{\A} \left. \int \wx^{-1} G \right|_{|u_j|=R(\wx,\wz)} \, d\tau \, d\wx \, d\widehat{u}^j \, d\wz \\ & \leq  c \, \|f\|_{\A} \int \wx^{-1} d_M((x',y',z'), (\wx, \wy, \wz))^{\A} G \, d\tau \, d\wx \, d\widehat{u}^j \, d\wz
\end{split}
\end{equation*}
Proceeding exactly as for the estimate of $I_1$ at the right face, we obtain
$$
|A| \leq c \, \|f\|_{\A} \, d_M((x,y,z), (x', y', z'))^{\A}.
$$
\ \\[2mm] \textbf{Estimates where the diagonal meets the front face:}
The asymptotics of the heat kernel near the diagonal meeting the front face 
are conveniently described using the following projective coordinates
\begin{align}
\eta^2=\frac{t-\wt}{(x')^2}, \ S =\frac{(x'-\wx)}{\sqrt{t-\wt \ }}, \ 
U= \frac{y'-\wy}{\sqrt{t-\wt \ }}, \ Z =\frac{x' (z'-\wz)}{\sqrt{t-\wt \ }}, \  x', \ 
y', \ z'.
\end{align}

In these coordinates tf is the face in the limit $|(S, U, Z)|\to \infty$, 
ff and td are defined by $x', \eta$, respectively. 
For the transformation rule of the volume form we compute
\begin{align*}
\beta^*(d\wt \dv(\wx, \wy, \wz))=h \cdot (x')^{m+2} \eta^{m+1} (1-\eta S)^f d\eta \, dS \, dU \, dZ,
\end{align*}

where $h$ is a bounded distribution on $\mathcal{M}^2_h$. 
Hence, using \eqref{XH} we find for any $f\in \ho$ after  cancellations
\begin{equation*}
\begin{split}
|I_4|\leq  c \, \|f\|_{\A} \, d_M((x,y,z), (x', y', z'))^{\A} \int \int_{M^+} \eta^{-1} h G \, d\eta \, dS \, dU \, dZ.
\end{split}
\end{equation*}

We see that the crude estimate \eqref{XH} leads to an a priori singular behaviour in $\eta$.
In order to overcome this difficulty note first the following transformation rules for the edge
vector fields $\V$
\begin{align*}
\beta^*(x'\partial_{x'})=-2\eta \partial_\eta + \frac{1}{\eta}\partial_S + Z\partial_Z + x'\partial_{x'}, \quad 
\beta^*(x'\partial_{y'}) = \frac{1}{\eta}\partial_U + x' \partial_{y'}, \quad \beta^*(\partial_{z'}) = \frac{1}{\eta}\partial_Z + \partial_{z'}.
\end{align*}

We find that the singular behaviour in $\eta$ comes from differentiation in $(S,U,Z)$.
Hence we might as well consider the case $\beta^*(XH)=(x'\eta)^{-2}\partial^2_S \beta^*H$. 
Estimates of the other cases follow along the same lines. 
For fixed $(\eta, U,Z)$ the integration region is $\{|S|\leq R(\eta,U,Z)\}$, where $|S|=R(\eta, U, Z)\equiv R$ corresponds to $(\wx,\wy,\wz)\in \partial M^+$. We simplify the expressions below by setting 
$$
\delta f := f(t-(x'\eta)^2, x,y,z)-f(t-(x'\eta)^2, x',y',z').
$$
Integration by parts then gives
\begin{equation*}
\begin{split}
I_4 &= \int  \delta f \int_{M^+} (x'\eta)^{-2}\partial^2_S (\beta^*H) h 
 (x')^{m+2} \eta^{m+1} (1-\eta S)^f d\eta \, dS \, dU \, dZ \\
&=   \left. \int \delta f  \int_{\partial M^+}  (x'\eta)^{-2}\partial_S (\beta^*H) h 
 (x')^{m+2} \eta^{m+1} (1-\eta S)^f \right|_{|S|=R(\eta, U, Z)} d\eta \, dU \, dZ \\
&-   \int\delta f  \int_{M^+}  (x'\eta)^{-2}\partial_S (\beta^*H) \partial_S (h  (1-\eta S)^f)
 (x')^{m+2} \eta^{m+1} d\eta \, dS \,  dU \, dZ\\
&=: I'_4+I''_4,
\end{split}
\end{equation*}

The asymptotic behaviour of the heat kernel implies $\partial_S (\beta^*H)=(x'\eta)^{-m}G$, 
where $G$ is bounded and vanishes to infinite order as $|(S,U,Z)|\to \infty$. 
Using the fact that $h$ is a smooth function of $\wx=x'(1-\eta S)$ we find after  cancellations,
that $I''_4$ does not have singular behaviour in $\eta$ and estimate

\begin{equation*}
\begin{split}
|I''_4|\leq C \, \|f\|_{\A} \, d_M((x,y,z), (x', y', z'))^{\A}.
\end{split}
\end{equation*}
For the estimation of $I'_4$, we note that for $|S|=R(\eta, U, Z)$ we have $(\wx,\wy,\wz)\in \partial M^+$ so that 
\begin{align*}
&d_M((x,y,z), (\wx,\wy,\wz)) = 2 d_M((x,y,z), (x', y', z')), \\
&d_M((x', y', z'),(\wx,\wy,\wz)) \geq d_M((x,y,z), (x', y', z')), \\
&d_M((x', y', z'),(\wx,\wy,\wz)) \leq 3 d_M((x,y,z), (x', y', z')),
\end{align*}
where the second and third relations follow from the first one using the triangle inequality. We now transform to coordinates which are valid near the 
left corner of the front face and also near the diagonal

\begin{align*}
\eta^2=\frac{t-\wt}{(x')^2}, \ s=\frac{\wx}{x'}, \ u=\frac{y'-\widetilde{y}}{x'}, \ x', \ y', \ z', \ \widetilde{z},
\end{align*}
where in these coordinates $(\eta, s, x')$ are the defining functions of tf, lf and ff respectively. The coordinates are valid whenever $(\eta, s)$ are bounded as $(t-\wt, x',\wx)$ approach zero. The volume $dU \, dZ$ transforms as follows
\begin{align*}
dU \, dZ & =  \eta^{-m+1} du \, dz.
\end{align*}
Hence we find
\begin{equation*}
\begin{split}
|I'_4|&\leq c \, \|f\|_{\A} \, d_M((x,y,z), (x', y', z'))^{\A}  \left. \int
\eta^{-m} G \right|_{|S|=R} d\eta \, du \, d\wz \\
&\leq c' \, \|f\|_{\A} \, \left. \int
\eta^{-m} G d_M((x', y', z'),(\wx,\wz,\wz))^{\A}  \right|_{|S|=R} d\eta \, du \, d\wz 
\end{split}
\end{equation*}

Now we proceed exactly as before in the estimation of $I_1$. 
We pass from $(u,\wz)$ to polar coordinates around $(0,z')$ with the 
radial function 
\begin{align*}
r(u,z',\wz)=\sqrt{|z'-\wz|^2 + |u|^2}.
\end{align*}
Observe the following relation
$$
\frac{d_M((x', y', z'),(\wx,\wz,\wz))}{x'r(u,z',\wz)} \leq C'
\frac{\sqrt{|S|^2+|U|^2+|Z|^2}}{\sqrt{|U|^2+|Z|^2}}.
$$
Consequently substituting $\sigma=\eta / r$ 
and using the fact that $G$ is vanishing to infinite order as $|(S,U,Z)|\to \infty$ and hence also as $\sigma \to 0$, 
we find
\begin{equation*}
\begin{split}
|I'_4| &\leq c'' \, \|f\|_{\A} \, (x')^\A 
\left. \int_0^{\infty} \int r^{-1+\A} \, \sigma^{-m} 
\left(\frac{\sqrt{|S|^2+|U|^2+|Z|^2}}{\sqrt{|U|^2+|Z|^2}}\right)^{\A} \, G
\right|_{|S|=R} \, 
dr \, d\sigma \\
&\leq c''' \, \|f\|_{\A} d((x,y,z),(\wx,\wy,\wz))^\A.
\end{split}
\end{equation*}

\ \\[2mm] \underline{\emph{Estimation of the H\"older difference in space 
without differentiation.}}\bigskip

In view of the H\"older norm for $\hho$ 
we also need to estimate the H\"older difference in space without differentiating 
the heat kernel. In other words we need to establish mapping properties 
of the heat operator as a convolution operator between $\ho$ and itself.
Here we use mean value theorem, \eqref{stoch-compl} 
and estimating in a manner similar to $I_3$ above
\begin{equation*}
\begin{split}
&e^{-t\Delta_g}*f(t,x,y,z) - e^{-t\Delta_g}*f(t,x',y',z') \\
&=  |x-x'|\int_0^t\int_{M} \partial_{\xi} H(t-\wt, \xi, y, z, \wx, \wy, \wz) 
\left[ f(\wt, \wx, \wy, \wz)- f(\wt, x, y, z)\right] \,  \dv \, d\wt \\
&+\,  |y-y'|\int_0^t\int_{M} \partial_{\gamma} H(t-\wt, x', \gamma, z, \wx, \wy, \wz) 
\left[ f(\wt, \wx, \wy, \wz)- f(\wt, x, y, z)\right] \,  \dv \, d\wt \\
&+\, |z-z'|\int_0^t\int_{M} \partial_{\zeta} H(t-\wt, x', y', \zeta, \wx, \wy, \wz) 
\left[ f(\wt, \wx, \wy, \wz)- f(\wt, x, y, z)\right] \,  \dv \, d\wt.
\end{split}
\end{equation*}

The estimates required here are much simpler than those for $I_3$ as we apply two  fewer derivatives to the heat kernel. For brevity
we omit the straightforward computations which consist only of lifting 
the integrand to $\mathscr{M}^2_h$ and checking the powers of defining functions 
at the various boundary faces. Using $|\rho_\ff \rho_\td G|\leq C \sqrt{t}$
we find
\begin{align}
 \label{sqrt-time1}
\|e^{-t\Delta_g}*f(t,x,y,z) - e^{-t\Delta_g}*f(t,x',y',z')\|_\infty 
\leq C \sqrt{t} \|f\|_\A d_M((x,y,z),(x',y',z'))^\A.
\end{align}

\subsection{Estimation of the H\"older difference in time}\label{ho-time} \ \\[-1mm]

Assume $t>t'$ without loss of generality. Consider first the case $(2t'-t)\geq 0$. 
Then we consider for any $X\in \{\Delta_g, \ x^{-1}\mathcal{V}'_e\}$ the difference
\begin{align*}
&X e^{-t\Delta_g}*f(t,p) - X e^{-t\Delta_g}*f(t',p) \\
=&\int_{2t'-t}^t\int_{M} X H(t-\wt, p,\widetilde{p})\left[ f(\wt, \widetilde{p})- f(\wt, p)\right] \,  \dv (\widetilde{p}) \, d\wt\\
-&\int_{2t'-t}^{t'}\int_{M} X H(t'-\wt, p,\widetilde{p})\left[ f(\wt, \widetilde{p})- f(\wt, p)\right] \,  \dv (\widetilde{p}) \, d\wt\\
+&\int_0^{2t'-t}\int_{M} \left[ X H(t-\wt, p,\widetilde{p}) -X H(t'-\wt, p, \widetilde{p})\right] 
 \left[ f(\wt, \widetilde{p})- f(\wt, p)\right] \,  \dv (\widetilde{p}) \, d\wt\\
=:& J_1+J_2+J_3,
\end{align*}
where again we have used stochastic completeness of the heat kernel. 
Note that in case of $X=\partial_t$, we can employ the heat 
equation 
\[
\partial_t e^{-t\Delta_g} * f = - \Delta_g (e^{-t\Delta_g} * f) + f. 
\]
Hence without loss of generality we can consider 
$X\in \{\Delta_g, \ x^{-1}\mathcal{V}'_e\}$.

We lift the heat kernel to a polyhomogeneous distribution $\beta^*H$ on the 
parabolic blowup of the heat space $\mathscr{M}^2_h$. 
The estimates in the interior of $\mathscr{M}^2_h$ are classical and hence we may assume that $\beta^*H$ is compactly supported in an open neighbourhood of the front face and estimate $J_1$ and $J_2$ near the various corners of ff.  
By Proposition \ref{heat-expansion}  we find
\begin{align}
\label{XH2}
\beta^*(X H)=(\rho_{\ff}\rho_{\td})^{-m-2} \rho_{\tf}^{\infty} G
\end{align}
where $G$ is a bounded polyhomogeneous distributions on $\mathscr{M}^2_h$.
We write down the estimates for $J_1$ and $J_3$. The second integral $J_2$ is estimated along the lines of $J_1$. 
For $J_3$ we employ the mean value theorem and find for some $\theta \in (t',t)$
\begin{equation*}
\begin{split}
 J_3 &= \int_0^{2t'-t}\int_{M} \left[ X H(t-\wt, p,\widetilde{p}) -X H(t'-\wt, p, \widetilde{p})\right] 
 \left[ f(\wt, \widetilde{p})- f(\wt, p)\right] \,  \dv (\widetilde{p}) \, d\wt \\
     &= |t-t'| \int_0^{2t'-t}\int_{M} \partial_{\theta}X H(\theta-\wt, p,\widetilde{p}) 
 \left[ f(\wt, \widetilde{p})- f(\wt, p)\right] \,  \dv (\widetilde{p}) \, d\wt.
\end{split}
\end{equation*}
\textbf{Estimates near the lower left corner of the front face:}
Let us assume that the heat kernel $H$ is compactly supported near the lower left corner of the front face. 
Its asymptotic behaviour is appropriately described in the following projective coordinates
\begin{align*}
\tau=\frac{t-\wt}{x^2}, \ s=\frac{\wx}{x}, \ u=\frac{y-\widetilde{y}}{x}, \ x, \ y, \ z, \ \widetilde{z},
\end{align*}
where in these coordinates $\tau, s, x$ are the defining functions of tf, lf and ff respectively. 
The coordinates are valid whenever $(\tau, s)$ are bounded as $(t-\wt,x,\wx)$ approach zero. 
For the transformation rule of the volume form we compute
\begin{align*}
\beta^*(d\wt \dv(\wx, \wy, \wz)) =  h \cdot x^{m+2} s^f d\tau \, ds \, du \, d\wz,
\end{align*}
where $h$ is a bounded distribution on $\mathscr{M}^2_h$. 
Hence, using \eqref{XH2} we arrive for any $f\in \ho$ after 
 cancellations at the estimate (note that near the left corner $x^2 \geq (t-\wt)$)
\begin{equation*}
\begin{split}
|J_1|&\leq C \, \|f\|_{\A} \, \int x^\A (d_M((x,y,z), (\wx, \wy, \wz))/x)^{\A} \, G \, d\tau \, ds \, du \, d\wz \\
&\leq C' \, \|f\|_{\A} \, \int x^{-2+\A} \sqrt{|1-s|^2 + |z-\wz|^2 + |u|^2}^{\ \A} \, G \, d\wt \, ds \, du \, d\wz \\
&\leq C'' \, \|f\|_{\A} \, \int^{t}_{2t'-t} (t-\wt)^{-1+\frac{\A}{2}} d\wt = C''' \, \|f\|_{\A} \, (t-t')^{\frac{\A}{2}}.
\end{split}
\end{equation*}

We estimate $J_3$ using the projective coordinates as above where $t$ is replaced by $\theta$. 
Using that for $\wt \in [0,2t'-t]$ we have $|\theta - \wt| \geq |t'-\wt|$, since we have assumed $(2t'-t)\geq 0$. Hence we find
\begin{equation*}
\begin{split}
|J_3|&\leq C \, \|f\|_{\A} \, (t-t') \int x^{-2+\A} (d_M((x,y,z), (\wx, \wy, \wz))/x)^{\A} G d\tau \, ds \, du \, d\wz \\
&\leq C' \, \|f\|_{\A} \, (t-t') \int x^{-4+\A} \sqrt{|1-s|^2 + |z-\wz|^2 + |u|^2}^{\ \A} G d\wt \, ds \, du \, d\wz \\
&\leq C'' \, \|f\|_{\A} \, (t-t') \int^{2t'-t}_0 (\theta-\wt)^{-2+\frac{\A}{2}} d\wt \\
&\leq C'' \, \|f\|_{\A} \, (t-t') \int^{2t'-t}_0 (t'-\wt)^{-2+\frac{\A}{2}} d\wt  \leq C''' \, \|f\|_{\A} \, (t-t')^{\frac{\A}{2}}.
\end{split}
\end{equation*}
\textbf{Estimates near the lower right corner of the front face:}
Let us assume that the heat kernel $H$ is compactly supported near the lower right corner of the front face. 
Its asymptotic behaviour is appropriately described in the following projective coordinates
\begin{align*}
\tau=\frac{t-\wt}{\wx^2}, \ s=\frac{x}{\wx}, \ u=\frac{y-\widetilde{y}}{\wx}, \ z, \wx, \ \wy, \ \widetilde{z},
\end{align*}
where in these coordinates $\tau, s, \wx$ are the defining functions of tf, rf and ff respectively. 
The coordinates are valid whenever $(\tau, s)$ are bounded as $(t-\wt,x,\wx)$ approach zero. 
For the transformation rule of the volume form we compute
\begin{align*}
\beta^*(d\wt \dv(\wx, \wy, \wz))=h \cdot \wx^{m+1} d\tau \, d\wx \, du \, d\wz,
\end{align*}
where $h$ is a bounded distribution on $\mathcal{M}^2_h$. 
By Corollary \ref{corr} we may write
\begin{align*}
\beta^*(X H) = \wx^{-m-2} s^{\A_0} G' + \wx^{-m-1} G'' + \beta^*(\partial_y LH),
\end{align*}
with bounded polyhomogeneous kernels $G',G''$. 
Hence, we find for any $f\in \ho$ after cancellations and integration by parts 
for the third component (we omit the boundary terms there which are estimated similarly)
\begin{equation*}
\begin{split}
|J_1|&\leq C \, \|f\|_{\A} \, \int \left(\wx^{-1+\A}s^{\A_0}G' + \wx^{\A}G'''\right) 
(d_M((x,y,z), (\wx, \wy, \wz))/x)^{\A} d\tau \, d\wx \, du \, d\wz \\
&\leq C' \, \|f\|_{\A} \, \int \left(\wx^{-3+\A}s^{\A_0}G' + \wx^{-2+\A}G'''\right) 
\sqrt{|1-s|^2 + |z-\wz|^2 + |u|^2}^{\ \A}  d\wt \, d\wx \, du \, d\wz,
\end{split}
\end{equation*}
for some polyhomogeneous kernel $G'''$.
Using the fact that $G', G'''$ vanishe to infinite order as $|u|\to \infty$ and $\wx \geq x$ near rf, we find
\begin{align*}
\int \left(\wx^{-1}s^{\A_0}G' + G''\right) 
\sqrt{|1-s|^2 + |z-\wz|^2 + |u|^2}^{\ \A}  d\wx \, du \, d\wz 
\leq x^{\A_0}\int_x^{\infty} \wx^{-1-\A_0} d\wx + c \leq c'.
\end{align*}
Consequently, using $\wx^2 \geq (t-\wt)$ we find
\begin{equation*}
\begin{split}
|J_1| &\leq C' \, \|f\|_{\A} \, \int \left(\wx^{-3+\A}s^{\A_0}G' + \wx^{-2+\A}G'''\right) 
\sqrt{|1-s|^2 + |z-\wz|^2 + |u|^2}^{\ \A}  d\wt \, d\wx \, du \, d\wz \\
&\leq C'' \, \|f\|_{\A} \, \int^{t}_{2t'-t} (t-\wt)^{-1+\frac{\A}{2}} d\wt 
\leq C''' \, \|f\|_{\A} \, (t-t')^{\frac{\A}{2}}.
\end{split}
\end{equation*}
We estimate $J_3$ using the projective coordinates as above where $t$ is replaced by 
$\theta$. Using $|\theta - \wt| \geq |t-t'|$, we find as before
\begin{equation*}
\begin{split}
|J_3|&\leq C \, \|f\|_{\A} \, (t-t') \int \left(\wx^{-3+\A}s^{\A_0}G' + \wx^{-2+\A}G'''\right) 
(d_M((x,y,z), (\wx, \wy, \wz))/x)^{\A} d\tau \, d\wx \, du \, d\wz \\
&\leq C' \, \|f\|_{\A} \, (t-t') \int \left(\wx^{-5+\A}s^{\A_0}G' + \wx^{-4+\A}G'''\right) 
\sqrt{|1-s|^2 + |z-\wz|^2 + |u|^2}^{\ \A}  d\wt \, d\wx \, du \, d\wz \\
&\leq C'' \, \|f\|_{\A} \, (t-t') \int^{2t'-t}_0 (\theta-\wt)^{-2+\frac{\A}{2}} d\wt 
\leq C''' \, \|f\|_{\A} \, (t-t')^{\frac{\A}{2}}.
\end{split}
\end{equation*}
\textbf{Estimates near the top corner of the front face:} 
Let us assume that the heat kernel $H$ is compactly supported near the top corner of the front face. 
Its asymptotic behaviour is appropriately described in the following projective coordinates
\begin{align*}
\rho=\sqrt{t-\wt}, \  \xi=\frac{x}{\rho}, \ \widetilde{\xi}=\frac{\widetilde{x}}{\rho}, 
\ u=\frac{y-\widetilde{y}}{\rho}, \ y, \ z, \ \widetilde{z},
\end{align*}
where in these coordinates $\rho, \xi, \widetilde{\xi}$ are the defining functions of the faces ff, lf and rf respectively. 
The coordinates are valid whenever $(\rho, \xi,\widetilde{\xi})$ are bounded as $(t-\wt,x,\wx)$ approach zero. 
For the transformation rule of the volume form we compute
\begin{align*}
\beta^*(d\wt \dv(\wx, \wy, \wz))=h \cdot \rho^{m+1} \, \widetilde{\xi}^f \,  d\rho \, d\widetilde{\xi}\, du\, d\wz.
\end{align*}
where $h$ is a bounded distribution on $\mathcal{M}^2_h$. 
Hence, using \eqref{XH2} we find for any $f\in \ho$ after  cancellations
\begin{equation*}
\begin{split}
|J_1|&\leq C \, \|f\|_{\A} \, \int \rho^{-1+\A} (d_M((x,y,z), (\wx, \wy, \wz))/\rho)^{\A} G d\rho \, d\widetilde{\xi} \, du \, d\wz \\
&\leq C' \, \|f\|_{\A} \, \int_0^{\sqrt{2(t-t')}} \rho^{-1 +\A} d\rho  = C''' \, \|f\|_{\A} \, (t-t')^{\frac{\A}{2}}.
\end{split}
\end{equation*}
We estimate $J_3$ using the projective coordinates as above where $t$ is replaced by $\theta$. 
Using $|\theta - \wt| \geq |t-t'|$, we find as before
\begin{equation*}
\begin{split}
|J_3|&\leq C \, \|f\|_{\A} \, (t-t') \int \rho^{-3+\A} (d_M((x,y,z), (\wx, \wy, \wz))/\rho)^{\A} G d\rho \, d\widetilde{\xi} \, du \, d\wz \\
&\leq C' \, \|f\|_{\A} \, (t-t') \int_{\sqrt{(t-t')}}^{\infty} \rho^{-3 +\A} d\rho  = C''' \, \|f\|_{\A} \, (t-t')^{\frac{\A}{2}}.
\end{split}
\end{equation*}
\textbf{Estimates where the diagonal meets the front face:} 
The asymptotics of the heat kernel near the diagonal meeting the front face 
is conveniently described using the following projective coordinates
\begin{align*}
\eta^2=\frac{t-\wt}{x^2}, \ S =\frac{(x-\wx)}{\sqrt{t-\wt \ }}, \ 
U= \frac{y-\wy}{\sqrt{t-\wt \ }}, \ Z =\frac{x (z-\wz)}{\sqrt{t-\wt \ }}, \  x, \ 
y, \ z.
\end{align*}

In these coordinates tf is the face in the limit $|(S, U, Z)|\to \infty$, 
ff and td are defined by $x, \eta$, respectively. 
For the transformation rule of the volume form we compute
\begin{align*}
\beta^*(d\wt \dv(\wx, \wy, \wz))=h \cdot x^{m+2} \eta^{m+1} (1-\eta S)^f d\eta \, dS \, dU \, dZ,
\end{align*}
where $h$ is a bounded distribution on $\mathcal{M}^2_h$. 
Hence, using \eqref{XH2} we find for any $f\in \ho$ after  cancellations
\begin{equation*}
\begin{split}
|J_1|&\leq C \, \|f\|_{\A} \, x^\A \, \int \eta^{-1+\A} (d_M((x,y,z), (\wx, \wy, \wz))/\eta x)^{\A} G d\eta \, dS\, dU\, dZ \\
&\leq C' \, \|f\|_{\A} \, x^\A \int \eta^{-1 +\A} \sqrt{|S|^2+|U|^2+|Z|^2} G d\eta \, dS\, dU\, dZ  \\
&\leq C'' \, \|f\|_{\A} \, x^\A \int_{0}^{\sqrt{\frac{2(t-t')}{x^2}}} \eta^{-1 +\A} d\eta = C''' \, \|f\|_{\A} \, (t-t')^{\frac{\A}{2}}.
\end{split}
\end{equation*}
We estimate $J_3$ using the projective coordinates as above where $t$ is replaced by $\theta$. 
Using $|\theta - \wt| \geq |t-t'|$, we find as before
\begin{equation*}
\begin{split}
|J_3|&\leq C \, \|f\|_{\A} \, (t-t') x^{-2+\A} \int \eta^{-3 +\A} (d_M((x,y,z), (\wx, \wy, \wz))/\eta x)^{\A} G d\eta \, dS\, dU\, dZ \\
&\leq C' \, \|f\|_{\A} \, (t-t') x^{-2+\A} \int \eta^{-3 +\A} \sqrt{|S|^2+|U|^2+|Z|^2} G d\eta \, dS\, dU\, dZ  \\
&\leq C'' \, \|f\|_{\A} \, (t-t') x^{-2+\A} \int_{\sqrt{\frac{t-t'}{x^2}}}^{\infty} \eta^{-3 +\A} d\eta = C''' \, \|f\|_{\A} \, (t-t')^{\frac{\A}{2}}.
\end{split}
\end{equation*}
This concludes the estimation near the diagonal assuming $2t'-t \geq 0$.
Finally, if $(2t'-t)<0$ then we consider for any 
$X\in \{\Delta_g, \ x^{-1}\V\}$ the difference
\begin{align*}
&X e^{-t\Delta_g}*f(t,p) - X e^{-t\Delta_g}*f(t',p) \\
=&\int_{0}^t\int_{M} X H(t-\wt, p,\widetilde{p})\left[ f(\wt, \widetilde{p})- f(\wt, p)\right] \,  \dv (\widetilde{p}) \, d\wt\\
-&\int_{0}^{t'}\int_{M} X H(t'-\wt, p,\widetilde{p})\left[ f(\wt, \widetilde{p})- f(\wt, p)\right] \,  \dv (\widetilde{p}) \, d\wt\\
=:& J'_1+J'_2,
\end{align*}
where we have used stochastic completeness of the heat kernel as before.
The integrals $J'_1,J'_2$ are estimated exactly as before the respective integrals  $J_1,J_2$, 
using that for $(2t'-t)<0$ we have $t,t'\leq 2|t-t'|$.

\ \\[2mm]\underline{\emph{Estimation of the H\"older difference in time
without differentiation.}}\bigskip

As before in the parallel case of estimating 
the H\"older difference in space without differentiation,
we also need to estimate the H\"older difference in time without differentiating 
the heat kernel. In other words we need to establish mapping properties 
of the heat operator as a convolution operator between $\ho$ and $\ho$.
Here we use mean value theorem, \eqref{stoch-compl} and estimate as in $J_3$ above $(t'<t)$
\begin{equation*}
\begin{split}
&e^{-t\Delta_g}*f(t,x,y,z) - e^{-t'\Delta_g}*f(t',x,y,z) \\
&=  (t-t')\int_0^t\int_{M} \partial_{\theta} H(\theta - \wt, x, y, z, \wx, \wy, \wz) 
\left[ f(\wt, \wx, \wy, \wz)- f(\theta, x, y, z)\right] \,  \dv \, d\wt.
\end{split}
\end{equation*}

The estimates required here are much simpler than those for $J_3$
as we estimate with two fewer derivatives applied to the heat kernel. For brevity
we omit the straightforward computations which consist only of lifting 
the integrand to $\mathscr{M}^2_h$ and checking the powers of defining functions 
at the various boundary faces. Using $|t-t'|\leq \sqrt{t} \, |t-t'|^{\frac{\A}{2}}$
we find
\begin{align}
 \label{sqrt-time2}
\|e^{-t\Delta_g}*f(t,x,y,z) - e^{-t'\Delta_g}*f(t',x,y,z) \|_{\A/2} 
\leq C \sqrt{t} \|f\|_\A |t-t'|^{\frac{\A}{2}}.
\end{align}

\subsection{Estimation of the Supremum Norm}\label{ho-sup}  \ \\[-1mm]

For any $X\in \{\Delta_g, \ x^{-1}\V\}$ stochastic 
completeness of the heat kernel yields
\begin{align*}
|X e^{-t\Delta_g}*f| \leq \|f\|_\A \int_0^t \int_M X H(t-\wt, p,\widetilde{p})
d_M(p,\widetilde{p})^\A \textup{dvol}_g(\widetilde{p}) \, d\wt.
\end{align*} 

As before, the case $X=\partial_t$ can be reduced to 
$X\in \{\Delta_g, \ x^{-1}\V\}$ using the heat equation. The integral above can be estimated against a constant in a straightforward ways using Theorem \ref{friedrichs-blowup} and Proposition \ref{heat-expansion}, in the projective coordinates near the various corners of the front face. The estimates near the diagonal are not so straightforward but follow the arguments for estimation of $I_1$ near the diagonal along the same lines.

In the case where $X = \textup{id}$, we cannot use stochastic completeness of the heat kernel, but the estimates are straightforward.
The proof of the Theorem \ref{boundedness} is now complete. \\
\end{proof} 

\begin{cor}
\label{t-gain-cor}
Under the conditions of Theorem \ref{boundedness} we also have for any $k\in \N_0$
\begin{align*}
&e^{-t\Delta_g}: \hhok \to \hhokk, \\
&e^{-t\Delta_g}: \hhok \to \sqrt{t} \hhok.
\end{align*}

\end{cor}

\begin{proof}
First note that by definition 
\begin{align*}
\hho \subset \dom (\Delta_{g,\max}) :=
\{f\in L^2(M,g) \mid \Delta_g f \in L^2(M,g)\}, 
\end{align*}
where $\Delta_g f\in L^2(M,g)$ is understood in the distributional sense.
The boundary conditions for the Friedrichs extension $\Delta_g$ have been 
discussed in (\cite{MazVer:ATO}, \S 2.5). The characterization in 
 (\cite{MazVer:ATO}, Proposition 2.5) implies in particular that any
continuous element of the maximal domain $\dom (\Delta_{g,\max})$ 
automatically lies in the Friedrichs domain $\dom (\Delta_{g})$. Hence 
$\hho \subset \dom (\Delta_{g})$ and similarly 
$\hhok \subset \dom (\Delta^k_{g})$. Consequently 
\begin{align*}
\forall f\in \hhok  \forall j\in \N_0, j\leq k:  \ \Delta^j_g e^{-t\Delta_g} * f = e^{-t\Delta_g} * \Delta^j_g f.
\end{align*}
Hence using \eqref{sqrt-time1} and \eqref{sqrt-time2} we deduce 
for $f\in \hhok$ and any $j\leq k$
\begin{align*}
 \|\Delta_g^j e^{-t\Delta_g}*f\|_\A = 
\|e^{-t\Delta_g}*\Delta_g^j f\|_\A \leq C \sqrt{t} \|f\|_{2k+\A}.
\end{align*}
This reduces the statement to Theorem \ref{boundedness}.
\end{proof}


\section{The contraction mapping argument} \label{sec:contraction}

In this section we outline a contraction mapping argument to prove the existence of local 
solutions to certain quasilinear parabolic equations.  For a fixed feasible edge metric $g$ we study the problem
\begin{equation} 
\label{GeneralPDE} 
\partial_t u(t) + \Delta_{g} u = B(u) + Q(u, \nabla u, \Delta u), 
\quad u(0) = 0,
\end{equation}

where roughly speaking $B$ is a not necessarily linear operator consisting of bounded terms of fixed regularity and $Q$ contains the quasilinear terms of lower regularity, which must satisfy certain quadratic estimates.  We remark that below we have not presented the most general possible argument here.  Using Schauder estimates for the homogeneous Cauchy problem for the heat operator it is possible to incorporate nonzero initial conditions.  We will not pursue this here. The main result of this section is

\begin{theorem} \label{thm:STE}
Consider a fixed feasible edge metric $g$ and the Cauchy problem 
\eqref{GeneralPDE} where the operators $B$ and $Q$ 
satisfy the following mapping properties $(k\in \N)$
\begin{enumerate}
\item $B: \hhok \rightarrow \hhok$ 
\item $Q: \hhok \to \hhokm$
\end{enumerate}
such that for any $u \in \hhok$ with $\|u\|_{2k+\A}<\mu$ 
there exists some $C(\mu)>0$ such that
\begin{enumerate}
\item $\| Bu - Bv\|_{2k+\A} \leq C(\mu) \|u-v\|_{2k+\A}, \| B(u)  \|_{2k+\A}\leq C(\mu)$,
\item $\|Qu - Qv\|_{2(k-1)+\A} \leq C(\mu) \max\{\nmhhok{u},\nmhhok{v} \} \nmhhok{u-v}$,
\item $\|Qu\|_{2(k-1)+\A} \leq C(\mu) \nmhhok{u}^2$.
\end{enumerate}
Then there is a unique local solution $u\in \hhok$ to equation \eqref{GeneralPDE} for some $T > 0$ sufficiently small.
\end{theorem}

\begin{proof}
We prove the statement using the contraction mapping principle on the Banach space $\hhok$.  
For positive real parameters $\mu$ and $T$ to be chosen below, set
\[ Z_{\mu, T} := \left\{ u \in \hhok: u(x,0)=0, \nmhhok{u} \leq \mu. \right\}. \]
In view of the inhomogeneous heat equation \eqref{GeneralPDE} we are led to define
\[ \Psi u := \int_0^t \int_M e^{-(t-s) \Delta_g}(p,\widetilde{p}) \big( B(u) + 
Q(u)\big)(s,\widetilde{p}) \,  \textup{dvol}_g(\widetilde{p}) \, ds . \]

First observe by the mapping properties of $B$ and $Q$ and from the Schauder estimates in the previous section that if $u \in \hhok$ then $\Psi u \in \hhok$.  It remains to show $\Psi$ is a contraction on $Z_{\mu,T}$. The proof of the contraction property consists of two parts. First we must show that $\Psi$ maps $Z_{\mu,T}$ to itself for some choice of $\mu$ and $T$, and then we prove the contraction estimate works for the same values of $\mu$ and $T$.  We begin by assuming that $\mu < 1$. Given $u \in Z_{\mu,T}$ we write
\begin{align*}
 \Psi u &= \int_0^t e^{-(t-s)  \Delta_g} \left( B(u) + Q(u) \right) ds \\
&= \int_0^t e^{-(t-s)  \Delta_g} \left( B(u) \right) ds + \int_0^t e^{-(t-s) \Delta_g} \left( Q(u) \right) ds \\
& =: (\Psi u)_1 + (\Psi u)_2.
\end{align*}

Applying the Schauder estimate to $(\Psi u)_2$ and using the properties of $Q$ we find 
\begin{align*}
\nmhhok{(\Psi u)_2} \leq C K \nmhhok{u}^2. 
\end{align*}
If we take $\mu < \min\{ \frac{1}{2CK}, 1 \}$, then for $u \in Z_{\mu,T}$
\[ \nmhhok{(\Psi u)_2} \leq C K \nmhhok{u}^2 \leq CK \mu^2 \leq \frac{\mu}{2}. \]
This fixes the value of $\mu$.  This value of $\mu$ persists if we take smaller values of $T$. Regarding $(\Psi u)_1$ we recall Corollary \ref{t-gain-cor}.  
This, combined with the choice $\mu <1$ implies for $u \in Z_{\mu,T}$ 
and $T<(\mu / 2CK)^2$
\[ \nmhhok{(\Psi u)_1} \leq C K \sqrt{T} \leq \frac{\mu}{2}. \]

Combining these two estimates now shows that $\Psi$ maps the ball $Z_{\mu, T}$ to itself,
provided $\mu,T>0$ are sufficiently small. We now check that $\Psi$ is a contraction.  For $u, v \in Z_{\mu, T}$, we must estimate the $\hhok$ norm of $\Psi u - \Psi v$.  
We use the triangle inequality to break up the integral into two pieces.
Using Corollary \ref{t-gain-cor} we find
\begin{align*}
\nmhhok{\int_0^t e^{-(t-s)  \Delta_g} (B(u)-B(v)) ds } 
\leq C K \sqrt{T} \nmhhok{u-v} < \frac{1}{2} \nmhhok{u-v}. 
\end{align*}
Similarly, using the quadratic estimates for $Q$
\begin{align*}
\nmhhok{\int_0^t e^{-(t-s)  \Delta_g} (Q(u)-Q(v)) ds} \leq C K \mu \nmhhok{u-v} < \frac{1}{2}\nmhhok{u-v} 
\end{align*}

We conclude that $\Psi$ is a contraction on $Z_{\mu, T}$.
By the Banach fixed point theorem, $\Psi$ has a unique fixed point in $Z_{\mu,T}$.  
This concludes the proof of Theorem \ref{thm:STE}.
\end{proof}

\section{Short-time existence of the Yamabe flow} \label{sec:yamabe}
The Yamabe flow is the following geometric evolution equation
\begin{equation} 
\label{YamabeFlow} 
\partial_t g(t) = -\scal(g(t)) \cdot g(t), 
\quad g(0) = \ginit,
\end{equation}
where $\scal(g(t))$ denotes the scalar curvature of the metric $g(t)$.  
See \cite{Brendle} for a recent survey of results related to this flow on compact manifolds.
The Yamabe flow preserves the conformal class of $\ginit$.  In particular, setting
\[ g(t) =  e^{2 u(p,t)} \cdot \ginit. \]
we transform the Yamabe flow to a scalar equation for $u$, where 
we write $\Delta$ for the Laplacian associated to $\ginit$
\begin{equation}\label{Yamabe-transf}
\partial_t u = - \frac{1}{2} e^{-2u} \left( 2(m-1) \Delta u - (m-2)(m-1) |\nabla u|^2 + \scal(\ginit) \right), 
\quad u(p,0) = 0,
\end{equation}

We now condition the Yamabe flow equation.  Using the exponential series we write for any $u\in \hhok$
$$e^{-2u} = 1 + u \sum_{n=0}^\infty \frac{(-2)^{n+1}}{(n+1)} \frac{u^n}{n!}=:1+uG(u),$$
where it is straightforward to check that the elements $e^{-2u}, G(u) \in \hhok$ are locally uniformly bounded. More precisely, for $\nmhhok{u} \leq \mu$ there exists a constant $C(\mu) >0$ such that $\nmhhok{e^{-2u}}, \nmhhok{G(u)} \leq C(\mu)$. Substituting $e^{-2u} = 1 + u G(u)$ into the Yamabe equation \eqref{Yamabe-transf} we obtain
\begin{align*}
\partial_t u &+ (m-1) \Delta u 
= - (m-1) u G(u) \Delta u + \frac{1}{2}(m-2)(m-1) e^{-2u} |\nabla u|^2- \frac{1}{2} e^{-2u} \scal(\ginit) , 
\end{align*}
Note that by a rescaling of the time variable we may remove the factor of $m-1$ multiplying the Laplacian.  We will instead study
\begin{align*}
\partial_t u &+  \Delta u = - u G(u) \Delta u + \frac{1}{2}(m-2) e^{-2u} |\nabla u|^2- \frac{1}{2(m-1)} e^{-2u} \scal(\ginit), 
\end{align*}
without further comment. Now set 
\begin{align*}
Q(u) &:= -  u \, G(u) \, \Delta u + \frac{1}{2}(m-2) e^{-2u} |\nabla u|^2, \\
B(u) &:= - \frac{1}{2(m-1)} e^{-2u} \, \scal(\ginit).
\end{align*}
In order to apply the short-time existence result of the previous section, we need only verify the mapping properties of $Q$ and $B$.  
\begin{lemma} \label{lemma:estforQ}
$Q: \hhok \rightarrow \hhokm$.
Furthermore, for any $u,u'\in \hhok$ with $\nmhhok{u},\nmhhok{u'} \leq \mu$, there 
exists a constant $C(\mu)>0$ such that 
\begin{align*}
&\|Q u - Q u'\|_{2(k-1)+\A} \leq C(\mu) \max\{ \nmhhok{u}, \nmhhok{u'}\} \nmhhok{u-u'}, \\
&\|Q u\|_{2(k-1)} \leq C(\mu) \nmhhok{u}^2.
\end{align*}
\end{lemma}

\begin{proof}
It suffices to consider $k=1$, the higher order cases treated ad verbatim.
The first mapping property is clear by the definition of the space $\hho$ and the fact that the product of functions in $\ho$ is again in $\ho$ by the Leibnitz formula.  Now if $u \in \hho$ with $\nmhho{u} \leq \mu$ then the first term of $Q(u)$ lies in $\ho$, and
\begin{align*}
\| - (m-1) u G(u) \Delta u\|_{\A} & \leq C(\mu) \nmhho{u} \|\Delta u\|_{\A}
\\ & \leq C(\mu) \nmhho{u}^2.
\end{align*}
For the second term, observe that for $u \in \hho$, $|\nabla u| \in \ho$ by the fact that $x^{-1} \mathcal{V}'_e u \in \ho$  
and $\nabla \in x^{-1} \mathcal{V}'_e$. Consequently,
\begin{align*}
\left\|\frac{1}{2}(m-2)(m-1) e^{-2u} |\nabla u|^2
\right\|_{\A} \leq C(\mu) \nmhho{u}^2. 
\end{align*}

The second estimate follows from the fact that $Q$ is locally Lipschitz.  
We estimate only one term, the remaining terms are similar.  In what follows we 
omit constants depending only on $m$.  First observe that fixing and $(p,t)$ we have
\begin{align*}
e^{-2u} |\nabla u|^2 - e^{-2u'} |\nabla u'|^2 &= 
(e^{-2u} - e^{-2u'}) |\nabla u|^2 + e^{-2u'} \left(|\nabla u|^2 - |\nabla u'|^2\right) \\
&= -2 e^{-2 \xi} (u - u') |\nabla u'|^2 + e^{-2u'} (|\nabla u|+ |\nabla u'|)(|\nabla u| - |\nabla u'|),
\end{align*}
where $\xi$ lies between $u(p,t)$ and $u'(p,t)$ by the mean value theorem.  
Estimating this expression using the triangle inequality in the $\ho$ norm we find that
\begin{align*}
&\nmho{e^{-2u} |\nabla u|^2 - e^{-2u'} |\nabla u'|^2} \\
& \hspace{1 cm} \leq C(\mu) \nmhho{u'}^2 \nmhho{u-u'}+ C(\mu) \max\{\nmhho{u},\nmhho{u'} \} \nmhho{u-u'} \\
& \hspace{1 cm} \leq C'(\mu) \max\{ \nmhho{u},\nmhho{u'} \} \nmhho{u-u'},
\end{align*}
as required.
\end{proof}

\begin{lemma} \label{lemma:estforB}
Assume that $\ginit$ is a feasible edge metric with 
$\scal(\ginit) \in \Lambda^{2k+\A}(M,\ginit)$. Then 
$B: \hhok \rightarrow \hhok$.
Furthermore, for any $u,u'\in \hhok$ with $\nmhhok{u},\nmhhok{u'} \leq \mu$, there 
exists a constant $C(\mu)>0$ such that 
\begin{align*}
 &\nmhhok{B(u) - B(u')} \leq C(\mu) \nmhhok{u-u'}, \\
 &\nmhhok{B(u)} \leq C(\mu).
\end{align*}
\end{lemma}
\begin{proof}
The H\"older space 
$\hhok$ preserved under composition and hence $B$ 
has the desired mapping property. Moreover 
for any $u \in \hhok$ with $\nmhhok{u} \leq \mu$
we have $\nmhhok{e^{-2u}} \leq C(\mu)$ for 
some constant $C(\mu)>0$. This proves the second estimate.
Regarding the first estimate, this follows as in Lemma \ref{lemma:estforQ} by the fact $B$ is locally Lipschitz.
\end{proof}

Using Lemmas \ref{lemma:estforQ} and \ref{lemma:estforB} we may now cite Theorem \ref{thm:STE} to obtain a short-time solution to equation \eqref{Yamabe-transf} and hence equation \eqref{YamabeFlow}.  This finishes the proof of Theorem \ref{mainthm}.  

\def\cprime{$'$}
\providecommand{\bysame}{\leavevmode\hbox to3em{\hrulefill}\thinspace}
\providecommand{\MR}{\relax\ifhmode\unskip\space\fi MR }
\providecommand{\MRhref}[2]{%
  \href{http://www.ams.org/mathscinet-getitem?mr=#1}{#2}
}
\providecommand{\href}[2]{#2}

\end{document}